\documentclass[11pt, a4paper]{amsart}

\title{Isomorphisms of tensor algebras arising from weighted partial systems}

\keywords{Tensor algebra; weighted partial system; Markov operators; C*-correspondence; non-deterministic dynamics; classification of non-self-adjoint operator algebras}

\subjclass[2000]{Primary: 47L30, 46K50, 46H20. 
Secondary: 46L08, 37A30}

\author{Adam Dor-On}
\address{Adam Dor-On,
Department of Pure Mathematics,
University of Waterloo, 
Waterloo, ON, Canada.} 
\email{adoron@uwaterloo.ca}

\thanks{The author is partially supported by 
an Ontario Trillium Scholarship}

\usepackage[a4paper, top=2.9cm, left=3cm, right=2.8cm, bottom=2.2cm]{geometry}
\usepackage{amssymb}
\usepackage{xspace}
\usepackage{todonotes}
\renewcommand{\MR}[1]{} 
\newcommand{\la}{\langle}
\newcommand{\ra}{\rangle}
\newcommand{\nn}{\mathbb{N}}

\newcommand{\cc}{\mathbb{C}}
\newcommand{\dd}{\mathbb{D}}
\newcommand{\tensor}{\mathcal{T}_+}

\DeclareMathOperator{\supp}{supp}

\DeclareMathOperator{\Ker}{Ker}
\DeclareMathOperator{\Ad}{Ad}

\DeclareMathOperator{\Span}{Span}
\DeclareMathOperator{\Fix}{Fix}
\DeclareMathOperator{\md}{md}
\newcommand{\Aa}{\mathcal{A}}
\newcommand{\Bb}{\mathcal{B}}
\newcommand{\Ff}{\mathcal{F}}
\newcommand{\Ll}{\mathcal{L}}
\newcommand{\Ss}{\mathcal{S}}
\newcommand{\Mm}{\mathcal{M}}
\newcommand{\Ee}{\mathcal{E}}

\newcommand{\Qq}{\mathcal{Q}}
\newcommand{\Tt}{\mathcal{T}}

\theoremstyle{plain}
\newtheorem{theorem}{Theorem}[section]
\newtheorem{lemma}[theorem]{Lemma}

\newtheorem{proposition}[theorem]{Proposition}
\newtheorem{corollary}[theorem]{Corollary}

\theoremstyle{definition}
\newtheorem{defi}[theorem]{Definition}
\newtheorem{remark}[theorem]{Remark}
\newtheorem{notation}[theorem]{Notation}
\newtheorem{example}[theorem]{Example}
\numberwithin{equation}{section}

\begin{document}

\begin{abstract}
We continue the study of isomorphisms of tensor algebras associated to a C*-correspondences in the sense of Muhly and Solel. Inspired by by recent work of Davidson, Ramsey and Shalit, we solve isomorphism problems for tensor algebras arising from weighted partial dynamical systems. We provide complete bounded / isometric classification results for tensor algebras arising from weighted partial systems, both in terms of the C*-correspondences associated to them, and in terms of the original dynamics. We use this to show that the isometric isomorphism and algebraic / bounded isomorphism problems are two distinct problems, that require separate criteria to be solved. Our methods yield alternative proofs to classification results for Peters' semi-crossed product due to Davidson and Katsoulis and for multiplicity-free graph tensor algebras due to Katsoulis, Kribs and Solel.
\end{abstract}

\maketitle

\section{Introduction}

Non-self-adjoint operator algebras associated to dynamical / topological / analytic objects and their classification via these objects have been the subject of study by many authors for almost 50 years, beginning with the work of Arveson \cite{Arveson-measure-preserving} and Arveson and Josephson \cite{Arveson-Josephson}.

The main theme of this line of research, as is the main theme of this paper, is to identify the extent of which the dynamical objects classify their associated non-self-adjoint operator algebras. In this paper, we shall mainly focus on classification of non-self-adjoint tensor operator algebras arising from a single C*-correspondence over a commutative C*-algebra, although analogous work has been done in related contexts \cite{Davidson-Kakariadis-non-commutative-conj, Kakariadis-Katsoulis-multivar-nonco-isom, DRS, DRS-2, StochMat, Gurevich-SPS, Hartz-pick-isomorphism, Kakariadis-Shalit, Ramsey-multivar} to mention only some.

\subsection{History}
The origin of this line of research is the work of Arveson \cite{Arveson-measure-preserving} and Arveson and Josephson \cite{Arveson-Josephson}. Peters \cite{Peters-original} then continued this investigation where he introduced his semicrossed product algebra, and generalized the Arveson--Josephson classification. Hadwin and Hoover \cite{ConjTrans} improved Peters' classification by removing some of the restrictions on fixed points. For decades it was unknown if a restrictions on fixed points was necessary, and the problem of whether or not it was possible remove all restrictions came to be known as the conjugacy problem.

In a sequence of papers \cite{Muhly-Solel-Tensor-Rep, Muhly-Solel-Wold, Muhly-Solel-Morita}, Muhly and Solel established a non-commutative generalization of function theory, and in \cite{Muhly-Solel-Morita} they initiated a program of classifying all tensor algebras arising from C*-correspondences. In \cite{Muhly-Solel-Morita}, Muhly and Solel introduce a notion of aperiodicity of a C*-correspondence and classify tensor algebras arising from aperiodic C*-correspondences up to isometric isomorphism (See Section 5 in \cite{Muhly-Solel-Morita}).

Some of the first successful attempts to address the case of periodic C*-correspondences came from tensor algebras associated to countable directed graphs. Solel \cite{Solel-Quiver} (for finite directed graph and isometric isomorphisms) and Katsoulis and Kribs \cite{IsoDirGrAlg} (for countable directed graphs and algebraic / bounded isomorphisms) independently introduced methods of representations into upper triangular $2 \times 2$ matrices to solve isomorphism problems for tensor algebras associated to countable directed graphs, and show that for two countable directed graphs $G$ and $G'$, their tensor algebras $\tensor(G)$ and $\tensor(G')$ are isometrically / bounded isomorphic if and only if $G$ and $G'$ are isomorphic, and when $G$ (or $G'$) has either no sinks or no sources, this is equivalent to the existence of an algebraic isomorphism. We will provide an alternative proof of this result when the graphs $G$ and $G'$ are finite and multiplicity-free (See Corollary \ref{cor:graph-mf}).

The conjugacy problem was finally solved by Davidson and Katsoulis \cite{IsoConjAlg}, by adapting the methods of Hadwin and Hoover in \cite{ConjTrans}, and the methods of representations into upper triangular $2 \times 2$ matrices in \cite{IsoDirGrAlg, Solel-Quiver}. They prove that for two continuous maps $\sigma : X \rightarrow X$ and $\tau : Y \rightarrow Y$ on locally compact spaces $X$ and $Y$ respectively, the semicrossed product algebras $C_0(X) \times _{\sigma} \mathbb{Z}_+$ and $C_0(Y) \times _{\tau} \mathbb{Z}_+$ are algebraically / bounded / isometrically isomorphic if and only if $\sigma$ and $\tau$ are conjugate. We will provide an alternative proof to this result in the case when $X$ and $Y$ are compact (See Corollary \ref{cor:conjugacy-problem}).

In \cite{Multivar}, Davidson and Katsoulis associated an operator algebras $\Aa(X,\sigma)$ to a multivariable system $\sigma = (\sigma_1,...,\sigma_d)$ of continuous maps $\sigma_i : X \rightarrow X$ on a locally compact space $X$. They show that if $\sigma$ and $\tau$ are multivariable systems on locally compact $X$ and $Y$ respectively, and if $\Aa(X,\sigma)$ and $\Aa(Y,\tau)$ are algebraically isomorphic, then there is a homeomorphism $\gamma : X \rightarrow Y$ and an open cover $\{ U_{\alpha} | \alpha \in S_n \}$ of $X$ such that for each $\alpha \in S_n$ we have $\gamma^{-1}\tau_i \gamma |_{U_{\alpha}} = \sigma_{\alpha(i)} |_{U_{\alpha}}$. This equivalence relation between two multivariable systems is called \emph{piecewise conjugacy}.

A strong converse showing that if $(X,\sigma)$ and $(Y,\tau)$ are piecewise conjugate then $\Aa(X,\sigma)$ and $\Aa(Y,\tau)$ are \emph{isometrically} isomorphic was achieved under various assumptions (See Theorem 3.25 in \cite{Multivar}), but the general converse remains an open problem to this day (See Conjecture 3.31 and Theorem 3.33 in \cite{Multivar}).

Davidson and Roydor \cite{Davidson-Roydor-topo-graphs} then extend the ideas in \cite{Multivar} to the context of compact topological graphs in the sense of Katsura \cite{Katsura-topological-graphs}, which includes finite directed graphs, multivariable dynamics on compact spaces and more. Davidson and Roydor associate a tensor algebra $\tensor(\Ee)$ to every topological graph $\Ee$ in a way that generalizes the constructions mentioned above, and show that if $\Ee$ and $\Ff$ are two topological graphs with bounded isomorphic tensor algebras, then the topological graphs are locally conjugate, in the sense of Definition 4.3 in \cite{Davidson-Roydor-topo-graphs}. 

Again a strong converse going through isometric isomorphism was established when the dimension of the vertex space is at most 1. We will provide a proof that certain multiplicity free topological graphs also satisfy this strong converse (See Corollary \ref{cor:multip-free-part-sys}).

\subsection{General description}

In this article we provide classification results for tensor algebras arising from weighted partial systems (WPS for short). Our objective is to show that WPS yield tensor algebras which are still completely classifiable up to bounded / isometric isomorphisms, while covering as many examples of such classification results, for instance those for multiplicity free finite directed graphs \cite{IsoDirGrAlg, Solel-Quiver} and for Peters' semi-crossed product \cite{IsoConjAlg}. 

A weighted partial system on a compact space $X$ to is a pair $(\sigma,w)$ of $d$-tuples $(\sigma_1,...,\sigma_d)$ and $(w_1,...,w_d)$ of partially defined continuous functions $\sigma_i : X_i \rightarrow X$ and $w_i : X_i \rightarrow (0,\infty)$ for $X_i$ clopen. WPS generalize many classical constructions such as non-negative matrices, continuous function on a compact space, multivariable systems, distributed function systems, graph directed systems and more. 

To each WPS $(\sigma,w)$ we associate a multiplicity free topological quiver (in the sense of \cite{Muhly-Tomforde}) that encodes some information on it. This topological quiver gives rise to a C*-correspondence $C(\sigma,w)$, as constructed in \cite{Muhly-Tomforde}. We completely characterize these C*-correspondences up to unitary isomorphism and similarity, in terms of conjugacy relations between the WPS that we call \emph{branch-transition conjugacy} and \emph{weighted-orbit conjugacy} respectively.

We then associate a tensor algebra $\tensor(\sigma,w)$ to $C(\sigma,w)$ as one usually does for general C*-correspondences \cite{Muhly-Solel-Tensor-Rep, Muhly-Solel-Wold, Muhly-Solel-Morita} and classify these tensor algebras up to isometric / bounded isomorphism and in some cases up to algebraic isomorphism.

The following are our main results (See Theorems \ref{theorem:alg-bdd-main} and \ref{theorem:isom-main}). Suppose $(\sigma,w)$ and $(\tau,u)$ are WPS over compact $X$ and $Y$ respectively. 
\begin{enumerate}
\item
$\tensor(\sigma,w)$ and $\tensor(\tau,u)$ are isometrically isomorphic if and only if $C(\sigma,w)$ and $C(\tau,u)$ are unitarily isomorphic if and only if $(\sigma,w)$ and $(\tau,u)$ are branch-transition conjugate.
\item
$\tensor(\sigma,w)$ and $\tensor(\tau,u)$ are bounded isomorphic if and only if $C(\sigma,w)$ and $C(\tau,u)$ are similar if and only if $(\sigma,w)$ and $(\tau,u)$ are weighted-path conjugate. If in addition the clopen sets $X_i$ (which are the domains of each $\sigma_i$) cover $X$, the above is equivalent to having an algebraic isomorphism between $\tensor(\sigma,w)$ and $\tensor(\tau,u)$.
\end{enumerate}

The solution to these isomorphism problems requires an adaptation of a new method in the analysis of character spaces due to Davidson, Ramsey and Shalit in \cite{DRS} used in the solution of isomorphism problems of universal operator algebras associated to tuples of operators subject to homogeneous polynomial constraints.

One of the main thrusts of this paper is the use of these classification results to show that the (completely) isometric isomorphism and algebraic / (completely) bounded isomorphism problems are distinct in the sense that they require separate criteria to be solved (See Example \ref{ex:isometric-vs-bounded}).

\subsection{Structure of paper}
This paper contains eight sections, including this introductory section. 

In Section \ref{sec:Hilb-GNS-quiver} we discuss some preliminary theory of Hilbert C*-correspondences, discuss three different notions of isomorphisms of C*-correspondences called unitary isomorphism, similarity and isomorphism, and construct the GNS C*-correspondence associated to a completely positive map in the sense of Chapter 5 of \cite{Lance}, and quiver C*-correspondences of a topological quiver in the sense of \cite{Muhly-Tomforde}.

In Section \ref{sec:WPS} we introduce the notion of a weighted partial system, and define three different notions of conjugacy between WPS called branch-transition conjugacy, weighted-orbit conjugacy and graph conjugacy. We then use Section \ref{sec:Hilb-GNS-quiver} to associate a C*-correspondence to every WPS in such a way that the three conjugacy relations above correspond to unitary isomorphism, similarity and isomorphism between the C*-correspondences. We give examples that show that these three conjugacy relations are distinct.

In Section \ref{sec:tensor-algebras} we discuss the general theory of tensor algebras arising from C* correspondences, and develop the theory of semi-graded isomorphisms by building up on ideas from Section 5 of \cite{Muhly-Solel-Morita} and Section 6 of \cite{StochMat}.

In Section \ref{sec:univ-prop-auto-cont} we show that tensor algebras arising from the C*-correspondence of a WPS possess a certain universal property by using tools from Section 3 of \cite{ExelCroPro}, and find conditions on the WPS to ensure any isomorphism onto the tensor algebra of a WPS is automatically continuous, by applying ideas from \cite{Donsig-Hudson-Katsoulis}.

In Section \ref{sec:character-space} we compute the character space of a tensor algebra associated to a WPS by adapting the methods of \cite{ConjTrans}, and provide a useful characterization of semi-gradedness in terms of the character space.

In Section \ref{sec:main-results} we reduce the general isomorphism problem to the problem on semi-graded isomorphisms, and using Sections \ref{sec:character-space}, \ref{sec:univ-prop-auto-cont}, \ref{sec:tensor-algebras} and \ref{sec:WPS} in tandem with a new character space technique due to \cite{DRS}, we conclude our main classification results.

Finally, in Section \ref{sec:app-comp} we compare and apply our our theorems to other tensor algebra constructions arising from non-negative matrices, single variable dynamics and partial systems with disjoint graphs, and show how in some cases we can apply our results to recover some previously obtained results in the literature.

\section{Hilbert C*-correspondences: GNS and quiver correspondences} \label{sec:Hilb-GNS-quiver}

In this section we give the basic definitions for Hilbert C*-correspondences, and continue with defining three notions of isomorphism between them. We then discuss two constructions of C*-correspondences, one associated to completely positive maps on C*-algebras, and another associated to topological quivers in the sense of \cite{Muhly-Tomforde}.

\subsection{Hilbert C*-correspondences}
We assume that the reader is familiar with the basic theory of Hilbert C*-modules, which can be found in \cite{Lance, Manu, Paschke-inner-prod}.
We only give a quick summary of basic notions and terminology as we proceed, so as to clarify our conventions.

Let $\Aa$ be a C*-algebra and $E$ a Hilbert C*-module over $\Aa$. We denote by $\Ll(E)$ the collection of adjointable operators on $E$. If in addition $E$ has a left $\Aa$-module structure given by a *-homomorphism $\phi : \mathcal{A} \rightarrow \mathcal{L}(E)$, we call $E$ a Hilbert C*-correspondence over $\Aa$. We often suppress notation and write $a \cdot \xi : =\phi(a)\xi$.

A key notion of C*-correspondences is the internal tensor product. If $E$ is a C*-correspondence over $\Aa$ with left action $\phi$, and $F$ is a C*-correspondence over $\Aa$ with left action $\psi$, then on the algebraic tensor product $E\otimes_{alg}F$ one defines an $\Aa$-valued pre-inner product satisfying $\la x_1 \otimes y_1 , x_2 \otimes y_2 \ra = \la y_1 , \psi(\la x_1 , x_2 \ra) y_2 \ra$ on simple tensors. The usual Hausdorff completion process with respect to the norm induced by this inner product, yields the internal Hilbert C*-module tensor product of $E$ and $F$, denoted by $E\otimes F$ or $E\otimes_{\psi}F$, which is a C*-correspondence over $\Aa$ with left action $\phi \otimes Id_{F}$.

We now introduce certain types of morphisms between $C^*$-correspondences that arise naturally in the context of tensor algebras.

\begin{defi}
Let $E$ and $F$ be $C^*$-correspondences over the C*-algebras $\Aa$ and $\Bb$ respectively, let $\rho: \Aa \rightarrow \Bb$ be a *-isomorphism. Then we define the following:
\begin{enumerate}
\item
A $\rho$-bimodule map $V: E \rightarrow F$ is a map satisfying $V(a\xi b) = \rho(a) V(\xi) \rho(b)$.
\item
A \emph{bounded} $\rho$-bimodule map $V : E \rightarrow F$ is called a $\rho$-correspondence map.
\item
A $\rho$-bimodule map $V: E\rightarrow F$ is called $\rho$-adjointable if there exists $\rho^{-1}$-bimodule adjoint $V^* : F \rightarrow E$. That is, for $\xi \in F$ and $\eta \in E$,
$$
\la V^*(\xi) , \eta \ra = \rho^{-1} ( \la \xi, V(\eta) \ra )
$$
\end{enumerate}
\end{defi}

We note in passing that a $\rho$-adjointable map $V : E \rightarrow F$ is automatically bounded by the Uniform Boundedness Principle, where the $\rho$-adjoint $V^* : F \rightarrow E$ is a $\rho^{-1}$-correspondence map.

Given a *-isomorphism $\rho : \Aa \rightarrow \Bb$ and a C*-correspondence $F$ over $\Bb$, we may define a C*-correspondence structure $F_{\rho}$ over $\Aa$ on the set $F$. For $a\in \Aa$ and $\xi \in F$, we define left and right actions given by
$$
a \cdot \xi := \rho(a) \xi \ \ \text{and} \ \ \xi \cdot a := \xi \rho(a)
$$
and $\Aa$-valued inner product, given for $\xi,\eta \in F$ by
$$
\la \xi, \eta \ra_{\rho} = \rho^{-1}(\la \xi, \eta \ra)
$$
This construction turns $F$ into a C*-correspondence over $\Aa$, satisfies $(F_{\rho})_{\rho^{-1}} = F$ as C*-corrspondences over $\Bb$, and behaves well with respect to the internal tensor products. That is, if $F, F'$ are C*-correspondences over $\Bb$ and $\rho:\Aa \rightarrow \Bb$ is a *-isomorphism, then $(F\otimes_{\Bb} F')_{\rho}$ is naturally unitarily isomorphic to $F_{\rho} \otimes_{\Aa}F'_{\rho}$.

Next we show that tensor products of $\rho$-correspondence maps exist even when the maps are not necessarily adjointable. 

\begin{proposition}
Let $E, E'$ be C*-correspondences over $\Aa$ and $F,F'$ be C*-correspondences over $\Bb$. Suppose $V: E \rightarrow F$, $W:E' \rightarrow F'$ are $\rho$-correspondence maps for some *-isomorphism $\rho : \Aa \rightarrow \Bb$. Then there exists a unique $\rho$-correspondence map $V\otimes W : E\otimes E' \rightarrow F \otimes F'$ defined on simple tensors by $(V\otimes W)(\xi\otimes \eta) = V\xi \otimes W\eta$, such that $\|V \otimes W\| \leq \|V \| \cdot \| W\|$.
\end{proposition}

\begin{proof}
By the preceding discussion, looking at $F_{\rho}$, $F'_{\rho}$  and $(F\otimes_{\Bb} F')_{\rho} \cong F_{\rho} \otimes_{\Aa} F'_{\rho}$ instead, we may assume without loss of generality that $\Aa = \Bb$ and that $\rho = Id_{\Aa}$. Then, by item (1) of Subsection 8.2.12 of \cite{BlecherLeMerdy} the desired result follows.
\end{proof}

Hence, if $V: E \rightarrow F$ is a $\rho$-correspondence map, the maps $V^{\otimes n} : E^{\otimes n} \rightarrow F^{\otimes n}$ are $\rho$-correspondence maps, and are bounded with $\|V^{\otimes n} \| \leq \|V\|^n$. If in addition to that there exists $C > 0$ such that for all $n\in \nn$ we have $\|V^{\otimes n} \| \leq C$, we say that $V$ is \emph{tensor-power bounded}.

\begin{defi} \label{defi:isomorphisms-C-corresp}
Let $E$ and $F$ be $C^*$-correspondences over the C*-algebras $\Aa$ and $\Bb$ respectively, and let $\rho: \Aa \rightarrow \Bb$ be a *-isomorphism. 
\begin{enumerate}
\item
A $\rho$-correspondence map $V : E \rightarrow F$ is called a $\rho$-isomorphism if $V$ is bijective.
\item
A $\rho$-correspondence map $V : E \rightarrow F$ is called a $\rho$-similarity if $V$ is bijective and $V$ and $V^{-1}$ are tensor-power bounded.
\item
A map $U : E\rightarrow F$ is called a $\rho$-unitary if $U$ is a surjective isometric $\rho$-correspondence map.
\end{enumerate}
We will say that $E$ and $F$ are isomorphic / similar / unitarily isomorphic if there exist a *-isomorphism $\rho : \Aa \rightarrow \Bb$ and a $\rho$ - isomorphism / similarity / unitary $V: E \rightarrow F$ respectively.
\end{defi}

\begin{remark} \label{remark:adjointable-isomorphism-implies-unitary}
It turns out that $U : E\rightarrow F$ is $\rho$-unitary if and only if $U$ is $\rho$-adjointable and $U^*U = Id_E$ and $UU^* = Id_F$. In this case, we see that $U$ is a $\rho$-similarity. 

If $V$ is a $\rho$-\emph{adjointable} $\rho$-isomorphism, then $V^*V \in \Ll(E)$ is an $Id$-isomorphism and $V|V|^{-1}$ defines a $\rho$-unitary between $E$ and $F$. Hence, we note that in general, we do not assume that $\rho$-correspondence maps are $\rho$-adjointable. In fact, in Example \ref{ex:non-adjointable} we will see a $\rho$-similarity which is not $\rho$-adjointable.
\end{remark}

\subsection{GNS construction}
We now describe the general GNS (or KSGNS) construction associated to a completely positive map on a unital $C^*$-algebra. This construction is done in detail in \cite{Lance} and goes back to Paschke in \cite{Paschke-inner-prod}. 

Let $\Aa$ be a unital C*-algebra, and let $S$ be a completely positive map on $\Aa$. The GNS representation of $S$ is a pair $(GNS(S),\xi_{S})$ consisting of a Hilbert C*-correspondence $GNS(S)$ and a vector $\xi_{S} \in GNS(S)$ such that $S(a) = \la \xi_{S}, a\xi_{S} \ra$. 

$GNS(S)$ is defined as the C*-correspondence $\Aa \otimes_{S} \Aa$ which is the Hausdorff completion of the algebraic tensor product $\Aa \otimes \Aa$ with respect to the inner product and bimodule actions given respectively for $a,b,c,d\in \Aa$ by
$$
\la a\otimes b , c \otimes d \ra = b^*S(a^*c)d \ \ \text{and} \ \ a\cdot (b\otimes c) \cdot d = ab \otimes cd
$$
The vector of this correspondence is then given by $\xi_{S} = 1\otimes 1$, and clearly satisfies $S(a) = \la \xi_{S}, a\xi_{S} \ra$.

For an up-to-date account on the GNS construction and its associated Toeplitz, Cuntz-Pimsner and relative Cuntz-Pimsner algebras, see Section 3 of \cite{ExelCroPro}.

\subsection{Topological quivers and their C*-correspondences}
Another type of dynamical objects that turn up in our analysis are topological quivers, in the sense of \cite{Muhly-Tomforde}. See Subsection 3.3 in \cite{Muhly-Tomforde} for some of the classes of examples generalized by topological quivers.

We modify the definition of \cite{Muhly-Tomforde} to fit our settings and we note that it is the reversed of the one given in Definition 3.1 in \cite{Muhly-Tomforde}. Still the following choice of range and source maps is the one commonly used in higher rank graph algebras in the works of Kumjian and Pask in \cite{Kumj-Pask} and in topological graphs in the works of Katsura \cite{Katsura-topological-graphs, Topo-graph-Kats}. This choice has the additional advantage that composition of operators is identified with concatenation of edges in the usual, non-reversed way. 

\begin{defi} \label{defi:topological-quiver}
A \emph{topological quiver} is a quintuple $\Qq = (E^0,E^1,r,s,\lambda)$ such that $E^0$ and $E^1$ are compact Hausdorff spaces of vertices and edges respectively, such that $r:E^1 \rightarrow E^0$ is a continuous map, $s:E^1 \rightarrow E^0$ is an open continuous map and $\lambda = \{\lambda_v \}_{v\in E^0}$ are Radon measures on $E^1$ such that $\supp(\lambda_v) = s^{-1}(v)$ and for every $\xi \in C(E^1)$ the map $v\mapsto \int_{E^1}\xi(e)d\lambda_v(e)$ is in $C(E^0)$.
We call $\Qq = (E^0,E^1,r,s,\lambda)$ a \emph{topological graph} if in addition $s : E^1 \rightarrow E^0$ is a local homeomorphism, and $\lambda_v$ is counting measure on the discrete set $s^{-1}(v)$.
\end{defi}

Note that we do not assume that $E^0$ and $E^1$ are second countable as in Definition 3.1 in \cite{Muhly-Tomforde}, but for our purposes it will be enough to assume that $E^0$ and $E^1$ are compact. The definition of a \emph{topological graph} above is justified by Example 3.20 of \cite{Muhly-Tomforde}, as this gives rise to a topological graph in the sense of Katsura \cite{Katsura-topological-graphs}.

Now that we have defined the notion of topological quiver, we wish to construct a C*-correspondence from it, just as we have from a completely positive map on a C*-algebra. Let $\Qq = (E^0,E^1,r,s,\lambda)$ be a topological quiver. We let $\Aa = C(E^0)$ and define an $\Aa$-valued inner product and bimodule actions on $C(E^1)$ for $v\in E^0$, $\xi,\eta \in C(E^1)$ and $f,g\in C(E^0)$ by setting 
$$
\la \xi, \eta \ra(v) := \int_{s^{-1}(v)}\overline{\xi(e)}\eta(e)d\lambda_v(e)  \ \ \text{and} \ \ (f\cdot \xi \cdot g)(e): = f(r(e))\xi(e)g(s(e))
$$
With this inner product and bimodule actions, $C(E^1)$ becomes a C*-correspondence over $C(E^0)$, and we call it the quiver C*-correspondence associated to $\Qq$, which we denote by $C(\Qq)$. We note that when $\Qq$ is a topological graph in the sense of Katsura \cite{Katsura-topological-graphs}, this C*-correspondence coincides with the one that is usually associated to a topological graph, also in the sense of Katsura \cite{Katsura-topological-graphs}.

The following is a restatement of Definition 6.1 in \cite{Muhly-Tomforde} that fits our reversed definition of topological quiver.

\begin{defi} \label{defi:paths}
Let $\Qq = (E^0,E^1,r,s,\lambda)$ be a topological quiver. A \emph{path} in $\Qq$ is a finite sequence of edges $\mu = \mu_n...\mu_1$ with $r(\mu_i) = s(\mu_{i+1})$ for $1 \leq i \leq n-1$. We say that such a path has length $|\mu| :=n$. Let $E^n$ denote the collection of paths of length $n$. We extend the maps $r$ and $s$ to $E^n$ by setting $r(\mu) = r(\mu_n)$ and $s(\mu) = s(\mu_1)$. We endow $E^n$ with the topology inherited from $E^1 \times ... \times E^1$. Since $s$ is continuous and open and $r$ is continuous on $E^1$, we see that this persists when $s$ and $r$ are considered as maps on $E^n$.
\end{defi}

\begin{remark} \label{remark:tensor-iterate}
In Section 6 of \cite{Muhly-Tomforde} one defines Radon measures $\lambda_v^n$ inductively on $E^n$ which then makes the quintuple $\Qq^n=(E^0,E^n,r,s,\lambda^n)$ into a topological quiver in its own right. 

As mentioned in the discussion preceding Remark 6.3 in \cite{Muhly-Tomforde}, it turns out that the quiver C*-correspondence of $\Qq^n$ coincides with the $n$-th internal tensor iterate for the C*-correspondence of $\Qq$. In other words, we have that $C(\Qq)^{\otimes n}$ and $C(\Qq^n)$ are naturally (Id-)unitarily isomorphic as C*-correspondences over $C(E^0)$, via the map sending simple tensors $\xi_n \otimes... \otimes \xi_1 \in C(\Qq)^{\otimes n}$ to the function $\mu_n...\mu_1 \mapsto \xi_n(\mu_n)\cdot ... \cdot \xi_1(\mu_1)$ in $C(\Qq^n)$.
\end{remark}

In our work, we will mostly be dealing with multiplicity free topological quivers.

\begin{defi} \label{defi:mult-free}
Let $\Qq = (E^0,E^1,r,s,\lambda)$ be a topological quiver. We say that $\Qq$ is multiplicity free if for any two edges $e,e'\in E^1$, if $r(e) = r(e')$ and $s(e)=s(e')$ then $e = e'$.
\end{defi}

The advantage of multiplicity free topological quivers is that they can be identified as \emph{closed} subsets of $E^0 \times E^0$ in a canonical way. If $\Qq = (E^0,E^1,r,s,\lambda)$ is multiplicity free, we define a map $r \times s : E^1 \rightarrow E^0 \times E^0$ given by $(r \times s) (e) = (r(e),s(e))$. $r \times s$ is then an injective continuous map with closed range (as $E^1$ is compact) and so $E^1$ is homeomorphic to its image under $r \times s$, and $\Qq$ is isomorphic to the topological quiver $\Qq':= (E^0,(r \times s)(E^1), \pi_r, \pi_s, \lambda')$ where $\pi_r, \pi_s : E^0 \times E^0 \rightarrow E^0$ are given by $\pi_r(y,x) = y$ and $\pi_s(y,x) = x$ and $\lambda' = \{\lambda'_v\}$ is given by $\lambda'_v(E) = \lambda_v((r \times s)^{-1}(E))$. 

We will often think of multiplicity free topological quivers $\Qq = (E^0,E^1,r,s,\lambda)$ with $E^1 = (r \times s)(E^1)$ already a closed subspace of $E^0 \times E^0$, so that $r=\pi_r$, $s=\pi_s$ and $\lambda = \lambda'$.

\section{Weighted partial systems} \label{sec:WPS}

In this section we define the notion of weighted partial system and introduce two conjugacy relations between WPS. We show that the GNS correspondence and quiver correspondence of a weighted partial system coincide, and further characterize when such C*-correspondences of two weighted partial systems are unitarily isomorphic / similar in terms of the new conjugacy relations.

\subsection{Definitions and constructions}
We define the notion of weighted partial system, and then associate a completely positive map and a topological quiver to it. Examples of these constructs will appear in the next subsection.

\begin{defi}
Let $X$ be a compact space. A $d$-variable \emph{weighted partial system} (WPS for short) is a pair $(\sigma,w)$ where $\sigma = (\sigma_1,...,\sigma_d)$ is comprised of continuous maps $\sigma_i: X_i \rightarrow X$ where each $X_i$ is clopen in $X$, and $w = (w_1,...,w_d)$ is comprised of continuous non-vanishing weights $w_i : X_i \rightarrow (0,\infty)$. 
\end{defi}

When $w_i = 1$ for all $1 \leq i \leq d$, then the information on the weights is redundant, and in this case we replace $(\sigma,1)$ by $\sigma$ and call it a $d$-variable (clopen) \emph{partial system}. Partial systems were used under the name of "quantised dynamical systems" by Kakariadis and Shalit to classify tensor algebras associated to monomial ideals in the ring of polynomials in non-commuting variables, up to $Q$-$P$-local piecewise conjugacy (See Definition 8.6 and Corollary 8.12 in \cite{Kakariadis-Shalit}).

\begin{defi} \label{def:quiver-and-positive-op}
Let $(\sigma,w)$ be a $d$-variable WPS over a compact $X$.
\begin{enumerate}
\item
The positive operator associated to $(\sigma,w)$ is a positive linear map $P(\sigma, w) : C(X) \rightarrow C(X)$ given by
$$
P(\sigma, w)(f)(x) = \sum_{i : x \in X_i} w_i(x) f(\sigma_i(x))
$$

\item
The quiver associated to $(\sigma,w)$ is the quintuple $\Qq(\sigma,w) = (X,Gr(\sigma),r,s,P(\sigma,w))$ where $Gr(\sigma)$ is the (union) cograph of $\sigma$, i.e. the union of the cographs of $\sigma_i$ given by 
$$
Gr(\sigma) = \cup_{i=1}^d \{ \ (\sigma_i(x),x) \ | \ x\in X_i \ \}
$$
The range and source maps are given by $r(\sigma_i(x),x) = \sigma_i(x)$, $s(\sigma_i(x),x) = x$ and Radon measures
$$
P(\sigma,w)_x = \sum_{i: x \in X_i}w_i(x)\delta_{(\sigma_i(x),x)}
$$
\end{enumerate}
\end{defi}

We will refer to $Gr(\sigma)$ as the \emph{graph} of the system $\sigma$, even though it is actually the cograph of the system $\sigma$. Moreover, note that the graph we have constructed is multiplicity-free, even though the original system $(\sigma,w)$ need not be multiplicity-free. In other words, if for some $x\in X$ we have an index $1 \leq i \leq d$ such that $\sigma_i(x) = \sigma_j(x)$ for some other index $j\neq i$, then $\sigma$ is not multiplicity free, yet in $Gr(\sigma)$ we have $(\sigma_i(x),x) = (\sigma_j(x),x)$. Instead, part of the information on the multiplicity of $\sigma_i(x)$ is encoded in the measure $P(\sigma,w)_x$, depending on the weights $w_j(x)$ for these $j$ that satisfy $\sigma_j(x) = \sigma_i(x)$.

We also abused notation above and decided to denote both the positive map and the collection of Radon measures of the topological quiver of $(\sigma,w)$ in the same way, the reason being the 1-1 correspondence between positive maps $S : C(X) \rightarrow C(X)$ and uniformly bounded maps $x \mapsto P_x \in M(X)$ given by the relation $S(f)(x) = \int_X f(y) dP_x(y)$, as in Lemma 3.27 in \cite{ExelCroPro}.

For a WPS $(\sigma,w)$, it is the continuity of $w_i$ that guarantees that (integration against or application of) $P(\sigma,w)$ sends continuous functions to continuous functions, and indeed $P(\sigma,w)$ is a positive linear map / (uniformly) finite Radon measures, and $\Qq(\sigma,w)$ defined above is a topological quiver by Lemma 3.30 of \cite{ExelCroPro}.

\begin{remark}
When proving that $\Qq(\sigma,w)$ is a topological quiver, we use the fact that $w_i$ never vanishes. This ensures that for every $x\in X$ we have 
$$
\supp(P(\sigma,w)_x) = \{(\sigma_1(x),x),...,(\sigma_d(x),x)\} = s^{-1}(x)
$$
and so by Lemma 3.30 of \cite{ExelCroPro} we have that $\Qq(\sigma,w)$ is a topological quiver. When some $w_i$ vanishes at a point $x\in X$, one may arrive at the situation where $\Qq(\sigma,w)$ is not a topological quiver according to our definition, because it is then possible that $\supp(P(\sigma,w)_x)$ does not contain the edge $(\sigma_i(x),x)$ while $s^{-1}(x)$ always does. See Section 3.5 and Example 3.35 in \cite{ExelCroPro} for this phenomenon, and other complications that arise in the more general context of positive operators on $C_0(X)$ and associated topological quivers.
\end{remark}

See \cite{MarkovCalg} and Subsection 3.5 \cite{ExelCroPro} where multiplicity free (as in Definition \ref{defi:mult-free}) topological quivers are similarly associated to positive maps $P : C_0(X) \rightarrow C_0(X)$ with closed support, and where their crossed product C*-algebras are investigated.

\subsection{Subclasses of weighted partial dynamics} \label{subsec:examples} We show that weighted partial systems encompass many different classical dynamical objects. When they have simpler forms, we compute the associated topological quiver and positive map as in Definition \ref{def:quiver-and-positive-op}. 

\subsubsection{Non-negative matrices} \label{subsec:nnm}

If $A = [A_{ij}]$ is a non-negative matrix indexed by a finite set $\Omega$, we associate a $|\Omega|$-variable WPS $(\sigma^A,w^A)$ to it by specifying $\Omega^A_i : = \{ \ j \in \Omega \ | \ A_{ij} >0 \ \}$ and define $\sigma^A_i : \Omega^A_i \rightarrow \Omega$ by setting $\sigma^A_i(j) = i$, and $w^A_i(j) = A_{ij}$. Note that some $\sigma^A_i$ may be the empty set function. This way the graph of the WPS is given by $Gr(\sigma) = Gr(A) : = \{ \ (i,j) \ | \ A_{ij} >0 \ \}$, the Radon measures by $P(\sigma^A,w^A)_j = \sum_{i \in \Omega} A_{ij} \delta_{(i,j)}$ and the positive map $P(\sigma^A, w^A)$ by $P(\sigma^A, w^A)(f)(j) = \sum_{i \in \Omega}A_{ij}f(i)$. An account of the theory of non-negative matrices, Markov chains and their graph structure can be found in \cite{Seneta}.

\subsubsection{Finite directed graphs} \label{subsec:fdg}

Let $(V,E,r,s)$ be a directed graph with finitely many edges and vertices. We can regard every $v\in V$ as comprising a clopen subset $\{v\}$ of $V$, and each edge $e\in E$ as (the unique) map from $\{s(e)\}$ to $\{r(e)\}$. With $w_e = 1$, the collection $\sigma^E = \{e\}_{e\in E}$ becomes a (weighted) partial system. We then see that $Gr(\sigma^E)$, being the regular union in $V \times V$, yields the \emph{multiplicity free} directed graph associated to $(V,E,r,s)$, That is, $Gr(\sigma^E) : = \{ (r(e),s(e)) | e\in E \}$. Denote by $m_{w,v} = |s^{-1}(v) \cap r^{-1}(w)|$ the multiplicity of edges starting at $v$ and ending at $w$. Then the Radon measures are given by $P(\sigma^E)_v = \sum_{w : (w,v)\in Gr(\sigma^E)} m_{w,v} \delta_{(w,v)}$, and the positive map is given by 
$$
P(\sigma^E)(f)(v) = \sum_{w: (w,v) \in Gr(\sigma^E)} m_{w,v}f(w)
$$
We note that a directed graph $(V,E,s,r)$ with finitely many edges and vertices can be encoded as a non-negative matrix $A_E = [m_{w,v}]$ indexed by $V$, so that the topological quiver and positive maps for this example and the previous one coincide, when $A$ is a $\{0,1\}$-matrix.

Finally, we note that unless we started with a multiplicity free directed graph $(V,E,r,s)$, the cardinality of $Gr(\sigma^E)$ is strictly less than that of $E$, as the set of edges $Gr(\sigma^E)$ does not contain multiple edges between two fixed vertices, and we see that the information on multiplicity is instead encoded into the Radon measures.

\subsubsection{Partially defined continuous maps} \label{subsec:pdcf}
For a compact space $X$, a clopen subset $X' \subset X$ and a continuous map $\sigma : X' \rightarrow X$, then we have that $\sigma$ is a partial system. The positive map $P(\sigma)(f) = f \circ \sigma$ is a *-homomorphism on $C(X)$, and in fact, all *-homomorphisms on $C(X)$ arise in this way via the commutative Gelfand-Naimark duality. The graph of the partial system $\sigma$ is then just $Gr(\sigma) = \{ \ (\sigma(x),x) \ | \ x\in X' \ \}$, and the Radon measures are just Dirac measures $P(\sigma)_x = \delta_{(\sigma(x),x)}$.

\subsubsection{Multivariable systems} \label{subsec:ms}
When $\sigma = (\sigma_1,...,\sigma_d)$ is a $d$-tuple of continuous maps defined on \emph{all} of $X$, the graph of $\sigma$ is just the union of the graphs of $\sigma_i$ as in Definition \ref{def:quiver-and-positive-op}, but the Radon measures yield the simpler form $P(\sigma)_x = \sum_{i=1}^d\delta_{(\sigma_i(x),x)}$ for all $x\in X$. The positive operator associated to $\sigma$ is then given by $P(\sigma)(f)(x) = \sum_{i=1}^d f(\sigma_i(x))$.

\subsubsection{Distributed function systems} \label{subsec:dfs}
If $\sigma = (\sigma_1,...,\sigma_d)$ is a multivariable system on a compact metric space $X$, and $p=(p_1,...,p_d)$ are continuous non-vanishing probabilities in the sense that for each $x\in X$ we have that $\sum_{i=1}^dp_i(x) = 1$, we call $(\sigma,p)$ a distributed function system. The positive operator associated to $(\sigma,p)$ given by $P(\sigma)(f)(x) = \sum_{i=1}^d p_i(x) f(\sigma_i(x))$ yields a Markov-Feller operator on $C(X)$ in the sense of \cite{Zaharopol}, and this is a concrete way of creating examples of such Markov-Feller operators. 

When the $p_i$ are constant and each $\sigma_i$ is a 1-1 strict contraction, we call $(\sigma,p)$ a \emph{distributed iterated function} system. Markov-Feller operators were used explicitly, especially when $p_i = \frac{1}{d}$ for all $i$, by Hutchinson in \cite{Hutchinson} and Barnsley in \cite{Barnsley-Fractals, Barnsley-Superfractals} to construct certain invariant measures on self-similar sets coming from $\sigma$.

\subsubsection{Graph directed systems / Mauldin-Williams graphs} \label{subsec:gds}
Let $\mathcal{E}:= (V,E,r,s)$ be a directed graph with finitely many vertices and edges, $\{X_v\}_{v\in V}$ a (finite) set of compact metric spaces and $\{\sigma_e\}_{e\in E}$ a (finite) set of 1-1 strict contractions $\sigma_e : X_{s(e)} \rightarrow X_{r(e)}$. Then we call the data $(\mathcal{E},\{X_v\}_{v\in V}, \{\sigma_e\}_{e\in E})$ a \emph{graph directed system} or \emph{Mauldin-Williams graph}. If we set $X = \sqcup_{v\in V} X_v $, then $(\sigma_e)_{e\in E}$ becomes a partial system over $X$. 

See \cite{Mauldin-Urbanski-book}, where Mauldin-Williams graphs are used to construct self-similar sets and iterated limit sets. See also \cite{MW-Hausdorff-dim} where the Hausdorff dimension of such iterated limit sets is computed in some cases.

\subsection{Branch-transition conjugacy}

We next define two of the main conjugacy relations between WPS. One particular conjugacy relation that we call \emph{branch-transition} conjugacy, will turn out to arise from isometric isomorphism of the associated operator algebra.

We say that $(\sigma,w)$ and $(\tau,u)$ $d$-variable and $d'$-variable WPS over compact spaces $X$ and $Y$ respectively are conjugate if one is a homeomorphic image of the other up to some permutation. That is, $d=d'$ and there is a homeomorphism $\gamma: X \rightarrow Y$ and a permutation $\alpha \in S_d$ such that $\gamma^{-1}\tau_{\alpha(i)}\gamma = \sigma_i$ and $u_{\alpha(i)} \circ \gamma^{-1} = w_i$ for all $1\leq i \leq d$. 

Conjugation of WPS is the most rigid notion of conjugation that we shall encounter in this work. We define a weaker notion of conjugation, that loses some information about multiplicities and weights of the WPS.

For an $s$-variable WPS $(\tau,u)$ on a compact space $Y$ and a homeomorphism $\gamma:X \rightarrow Y$ denote $\tau^{\gamma} = \gamma^{-1}\tau \gamma: = (\gamma^{-1}\tau_1 \gamma,...,\gamma^{-1}\tau_s\gamma)$ and $u^{\gamma} = u \gamma = (u_1 \gamma , ...,u_s \gamma)$.

\begin{defi}
Let $\sigma$ and $\tau$ be partial systems on compact spaces $X$ and $Y$ respectively. We say that $\sigma$ and $\tau$ are \emph{graph conjugate} if there exists a homeomorphism $\gamma: X \rightarrow Y$ such that $Gr(\sigma) = Gr(\tau^{\gamma})$. Equivalently, there exists a homeomorphism $\gamma: X \rightarrow Y$ such that the map $\gamma \times \gamma : Gr(\sigma) \rightarrow Gr(\tau)$ is a homeomorphism.
\end{defi}

Some natural sets that arise while considering graphs of partial systems are the sets of points for which some of the maps in the system coincide and / or sets of points for which they "branch out".

\begin{defi}
Let $\sigma$ be a $d$-variable partial system.
\begin{enumerate}
\item
A point $x\in X$ is a \emph{branching point} for $\sigma$ if there is some net $x_{\lambda} \rightarrow_{\lambda} x$ and two indices $i,j\in \{1,...,d\}$ such that $x_{\lambda} \in X_i \cap X_j$ and $\sigma_i(x_{\lambda}) \neq \sigma_j(x_{\lambda})$ for all $\lambda$ while $\sigma_i(x) = \sigma_j(x)$. 
\item
An edge $e\in Gr(\sigma)$ is a \emph{branching edge} for $Gr(\sigma)$ if there are two nets $\{e_{\lambda}\}$ and $\{f_{\lambda}\}$ converging to $e$ such that $s(e_{\lambda}) = s(f_{\lambda})$ while $r(e_{\lambda}) \neq r(f_{\lambda})$ for all $\lambda$.
\end{enumerate}
\end{defi}

\begin{remark} \label{rem:graph-conj-preserves-branching-edges}
If $e$ is a branching edge, then by taking subnets if necessary, we see that $s(e)$ is a branching point. However, if $s(e)$ is a branching point, $e$ may not be a branching edge. Still, every branching point is the source of some branching edge. Thus, if two partial systems $\sigma$ and $\tau$ are graph conjugate via $\gamma$, then $\sigma$ and $\tau^{\gamma}$ have the same sets of branching points and branching edges.
\end{remark}

\begin{defi}
Let $\sigma$ be a $d$-variable partial system and $I \subset \{1,...,d\}$ a non-empty subset of indices. 
\begin{enumerate}
\item
The coinciding set of $I$ is the set
$$
C(I) = \{ \ x\in \cap_{i\in I}X_i \ | \ \sigma_{i}(x) = \sigma_{j}(x) \ \forall i,j\in I \ \}
$$
\item
$x\in X$ is a \emph{coinciding point} for $Gr(\sigma)$ if there is some $I\subset \{1,...,d\}$ with $|I| \geq 2$ such that $x\in C(I)$.
\end{enumerate}
We also denote $B(I) : = \partial C(I)$ the topological boundary of $C(I)$ inside $X$.
\end{defi}

For $I\subset \{1,...,d\}$, since the maps $\sigma_i$ of a partial system are defined on clopen sets $X_i$, we see that $\cap_{i\in I}X_i$ is clopen, so that both $C(I)$ and $B(I)$ are in fact (relatively) closed subsets of $\cap_{i\in I}X_i$.

We next characterize branching points and branching edges in terms of boundaries of coinciding sets.

\begin{proposition}
Let $\sigma$ be a $d$-variable (clopen) partial system. 
\begin{enumerate}
\item
$x\in X$ is a branching point if and only if for some $I\subset \{1,...,d\}$ we have $x\in B(I)$.
\item
$e\in Gr(\sigma)$ is a branching edge for $Gr(\sigma)$ if and only if $s(e) \in B(I)$ for some $I\subset \{1,...,d\}$ so that $r(e) = \sigma_i(s(e))$ for some (and hence all) $i\in I$.
\end{enumerate}
\end{proposition}

\begin{proof}
We first prove (1). If $x\in B(I)$ for some $I \subset \{ 1,...,d\}$, there exists a net $\{x_{\lambda}\}$ in $\cap_{i\in I}X_i$ converging to $x$ such that for every $\lambda \in \Lambda$ there exist $i_{\lambda} \neq j_{\lambda}$ in $I$ such that $\sigma_{i_{\lambda}}(x_{\lambda}) \neq \sigma_{j_{\lambda}}(x)$. By passing to a subnet, we may arrange that $i_{\lambda_1} = i_{\lambda_2}$ and $j_{\lambda_1} = j_{\lambda_2}$ for all $\lambda_1, \lambda_2 \in \Lambda$ and so $x$ is a branching point.

For the converse, if $x\in X$ is a branching point, let $i,j\in \{1,...,d\}$ be two distinct indices and $\{x_{\lambda}\}$ a net in $X_i \cap X_j$ converging to $x$ such that $\sigma_i(x_{\lambda}) \neq \sigma_j(x_{\lambda})$, while $\sigma_i(x) = \sigma_j(x)$. Then by taking $I=\{i,j\}$ we have that $x\in C(I)$, and the existence of the above net shows that $x\in B(I)$.

Next, we prove (2). $s(e) \in B(I)$ for some $I\subset \{1,...,d\}$ so that $r(e) = \sigma_i(s(e))$ for some (and hence all) $i\in I$. Then by the above we have a net $\{x_{\lambda}\}$ in $\cap_{i\in I} X_i$ converging to $s(e)$, and two distinct indices $i,j$ in $I$ such that $\sigma_i(x_{\lambda}) \neq \sigma_j(x_{\lambda})$ for all $\lambda$, while $\sigma_i(s(e)) = \sigma_j(s(e)) = r(e)$. Then the nets of edges $e_{\lambda} = (\sigma_i(x_{\lambda}), x_{\lambda})$ and $f_{\lambda} = (\sigma_j(x_{\lambda}), x_{\lambda})$ converges to $e$, and have the same sources, and different ranges for every $\lambda$.

Conversely, if we have two nets $\{e_{\lambda}\}$ and $\{f_{\lambda}\}$ converging to $e$ such that $s(e_{\lambda}) = s(f_{\lambda})$ while $r(e_{\lambda}) \neq r(f_{\lambda})$ for all $\lambda$, by taking subnets as necessary, we may assume that $r(e_{\lambda}) = \sigma_i(s(e_{\lambda}))$, and $r(f_{\lambda}) = \sigma_j(s(f_{\lambda}))$ while $r(e) = \sigma_i(s(e)) = \sigma_j(s(e))$ for $i,j$ distinct. By taking $I=\{i,j\}$ we see that $s(e) \in B(I)$, while $r(e) = \sigma_i(s(e))$, and we are done.
\end{proof}

For an edge $e\in Gr(\sigma)$, we denote $I(e,\sigma) = \{ \ i \ | \ \sigma_i(s(e)) = r(e), \ s(e) \in X_i \ \}$, which is the set of all indices of maps that send $s(e)$ to $r(e)$.

\begin{defi}
Let $(\sigma,w)$ be a WPS over $X$. The weight induced on the graph of $\sigma$ is a function $w : Gr(\sigma) \rightarrow (0,\infty)$ given for any edge $e=(y,x) \in Gr(\sigma)$ by
$$
w(e) = \sum_{i \in I(e,\sigma)}w_i(s(e)) = \sum_{i: \sigma_i(x) = y, \ x\in X_i}w_i(x)
$$
\end{defi}

\begin{proposition} \label{prop:weight-cont-bound}
Let $(\sigma,w)$ be a WPS over $X$. Then $w: Gr(\sigma) \rightarrow (0,\infty)$ is discontinuous at $e\in Gr(\sigma)$ if and only if $e$ is a branching edge for $Gr(\sigma)$. Moreover, $w: Gr(\sigma) \rightarrow (0,\infty)$ is bounded from above and from below.
\end{proposition}

\begin{proof}
$\Rightarrow$: If $e$ is not a branching edge for $\sigma$. Then there exist a neighborhood $U$ of $e$ inside $Gr(\sigma)$ such that for any $f\in U$ we have $I(f,\sigma) = I(e,\sigma)$. Hence, for $f\in U$ we have
$$
w(f) = \sum_{i \in I(f,\sigma)}w_j(s(f)) = \sum_{i \in I(e,\sigma)}w_j(s(f))
$$
so we see that $w$ is continuous at $e$ by continuity of $w_j$.

$\Leftarrow$: If $e \in Gr(\sigma)$ is a branching edge, without loss of generality, and perhaps by taking a subnet, there is a net $e_{\lambda} \rightarrow_{\lambda} e$ indexed by $\Lambda$ with $I := I(e_{\lambda_1},\sigma) = I(e_{\lambda_2},\sigma) \subsetneq I(e,\sigma)$ for all $\lambda_1, \lambda_2 \in \Lambda$. Hence we obtain that
$$
w(e_{\lambda}) = \sum_{i \in I}w_i(s(e_{\lambda})) \rightarrow \sum_{i \in I}w_i(s(e)) 
$$
by continuity of $w_i$ for all $1 \leq i \leq d$. Yet on the other hand,
$$
w(e) = \sum_{i \in I(e,\sigma)}w_i(s(e)) > \sum_{i \in I}w_i(s(e)) = \lim_{\lambda} w(e_{\lambda})
$$
since $I$ is a proper subset of $I(e,\sigma)$, and $w_i$ are bounded from below on the clopen sets $X_i$.

Finally, since for every $1 \leq i \leq d$ we have that $w_i$, being continuous on $X_i$ is bounded above by $M_i$ and below by $C_i$. Let $M = \max\{ M_1,...,M_d\}$ and $C = \min\{C_1,...,C_d\}$. Thus, if $e\in Gr(\sigma)$, there is some $i\in \{1,...,d\}$ with $\sigma(s(e)) = r(e)$ so that $w(e) \geq w_i(s(e)) \geq C > 0$, and of course $w(e) \leq |I(e,\sigma)| \cdot M \leq d \cdot M $, and we see that $w: Gr(\sigma) \rightarrow (0,\infty)$ is bounded below by $C$ and above by $d\cdot M$.
\end{proof}

\begin{defi}
Let $(\sigma,w)$ and $(\tau,u)$ be WPS on compact spaces $X$ and $Y$ respectively. We say that $(\sigma,w)$ and $(\tau,u)$ are \emph{branch-transition conjugate} if $\sigma$ and $\tau$ are graph conjugate via some homeomorphism $\gamma: X \rightarrow Y$ and we have that the weighted transition function $\frac{u^{\gamma}}{w}: Gr(\sigma) \rightarrow (0,\infty)$ from $w$ to $u^{\gamma}$ given by 
$$
\frac{u^{\gamma}}{w}(e) := \frac{u^{\gamma}(e)}{w(e)}
$$
is continuous at $e$ for any branching edge $e \in Gr(\sigma) = Gr(\tau^{\gamma})$.
\end{defi}

We interpret the above to mean that the discontinuities for $w$ and $u^{\gamma}$, which can only be at branching edges, are of the same proportions, so that the weighted transition function becomes continuous at every branched edge, and hence everywhere on $Gr(\sigma)$.

\begin{example} \label{example:cpc-distinct-btc}
Graph conjugacy does not imply branch-transition conjugacy, not even when the weights $w$ and $u$ are constant. If we take $X=[0,1]$ and $\sigma_1(x) = x$, $\sigma_2(x) = 0$, and pick two pairs of constant weights $u=(\frac{1}{2},\frac{1}{2})$ and $w = (\frac{1}{3},\frac{2}{3})$, then $(\sigma,w)$ and $(\sigma,u)$ are not branch-transition conjugate. Indeed, if $\sigma$ is graph conjugate to itself via $\gamma$, as $\gamma$ sends branching points to themselves, and $0$ is the only branching point, we then must have $\gamma(0) = 0$ which means that $\gamma$ must be non-decreasing. This means that
$$
\frac{u^{\gamma}}{w}(y,x) = \begin{cases} 
\frac{3}{2} & \text{if} \ x > 0 \ \& \ y=x \\ 
\frac{3}{4} & \text{if} \ x > 0 \ \& \ y=0 \\
1 & \text{if} \ x =0
\end{cases} 
$$ 
so that $\frac{u^{\gamma}}{w}$ is not continuous at the branching edge $e=(0,0)$, and so $(\sigma,w)$ and $(\sigma,u)$ are not branch-transition conjugate.
\end{example}

\begin{corollary} \label{corollary:radon-nikodym-continuity}
Let $(\sigma,w)$ and $(\tau,u)$ be $d$-variable and $s$-variable WPS on $X$ and $Y$ respectively. If $\sigma$ and $\tau$ are graph conjugate, there exists some $K \geq 1$ such that 
$$
\frac{1}{K} \leq \frac{u^{\gamma}}{w} \leq K
$$
If in addition, $(\sigma,w)$ and $(\tau,u)$ are branch-transition conjugate, then $\frac{u^{\gamma}}{w}$ is continuous on $Gr(\sigma) = Gr(\tau^{\gamma})$.
\end{corollary}

\begin{proof}
Without loss of generality we assume that $\gamma = Id_X$. Assuming $Gr(\sigma) = Gr(\tau)$, since both $w$ and $u$ are bounded from above and below by the last part of Proposition \ref{prop:weight-cont-bound}, we see that there is a $K \geq 1$ such that $\frac{1}{K} \leq \frac{u}{w} \leq K$.

Lastly, by Proposition \ref{prop:weight-cont-bound} again, both $w$ and $u$ are continuous at edges which are not branching points for $Gr(\sigma) = Gr(\tau)$, and by branch-transition conjugacy we see that $\frac{u}{w}$ is continuous on all of $Gr(\sigma)$.
\end{proof}

\subsection{C*-correspondences and tensor iterates}

In this subsection we identify the GNS correspondence and quiver correspondence of the positive operator / quiver of a WPS $(\sigma,w)$ respectively, and give a simple description of the tensor iterates in terms of iterates of the topological quiver, using Remark \ref{remark:tensor-iterate}. 

For the GNS correspondence $GNS(\sigma,w):= GNS(P(\sigma,w))$, for any $f,g,h,k \in C(X)$ the inner product formula and bimodule actions for simple tensors are given respectively by
$$
\la f \otimes g, h \otimes k \ra (x) =\sum_{i : x\in X_i} \overline{g(x)f(\sigma_i(x))} w_i(x) h(\sigma_i(x))k(x)  
\ \ \text{and} \ \ f \cdot (g \otimes h) \cdot k = fg \otimes hk
$$

Next, for the quiver correspondence of $\Qq(\sigma,w)$, which we denote by $C(\sigma,w): = C(\Qq(\sigma,w))$ the notation for weights of edges gives a nice formula for the Radon measures $P(\sigma,w)$ by $P(\sigma,w)_x = \sum_{s(e) =x}w(e)\delta_e$, so that for any $\xi, \eta \in C(Gr(\sigma))$ and $f,g\in C(X)$ we have left and right $C(X)$ actions given by
$$
(f \cdot \xi \cdot g)(e) = f(r(e)) \xi(e) g(s(e))
$$
and inner product 
$$
\la \xi, \eta \ra_{w}(x) = \sum_{s(e) = x} \overline{\xi(e)} w(e) \eta(e) = \sum_{i : x\in X_i} \overline{\xi(\sigma_i(x),x)} w_i(x) \eta(\sigma_i(x),x)
$$
We denote $f\odot g \in C(Gr(\sigma))$ the function given by $(f\odot g)(e) = f(r(e))g(s(e))$, which identifies well with the element $f\otimes g$ in $GNS(\sigma,w)$, as the following proposition demonstrates.

\begin{proposition} \label{proposition:GNS-computation}
Let $(\sigma,w)$ be a $d$-variable WPS on a compact space $X$. Then the map $f\otimes g \mapsto f \odot g$ uniquely extends to a unitary isomorphism between $GNS(\sigma,w)$ and $C(\sigma,w)$. In fact, the supremum norm on $C(Gr(\sigma))$ and the norm induced by the inner product on $C(\sigma,w)$ are equivalent.
\end{proposition}

\begin{proof}
First of all, we note that the norm $\|\cdot \|_w := \|\la \cdot, \cdot \ra_w^{\frac{1}{2}} \|$ defined on $C(Gr(\sigma))$ is equivalent to the supremum norm on it. Indeed, for $\xi \in C(Gr(\sigma))$ we have
$$
\sup_{i, \ x\in X_i}w_i(x)|\xi(\sigma_i(x),x)|^2 \leq \sup_{x} \sum_{i:x\in X_i} w_i(x)|\xi(\sigma_i(x),x)|^2 \leq d \cdot \sup_{i, \ x\in X_i}w_i(x)|\xi(\sigma_i(x),x)|^2
$$
Since for each $1 \leq i \leq d$ we have that $w_i$ is positive and continuous on $X_i$, and $X_i$ is compact, there exists $C >0$ such that for all $1\leq i \leq d$ and $x\in X_i$ we have $\frac{1}{C} \leq w_i(x) \leq C$, so that
$$
\frac{1}{C} \cdot \|F\|_{Gr(\sigma)}^2 \leq \| F \|_w^2 \leq d C \cdot \|F\|_{Gr(\sigma)}^2
$$
We show on finite sums of simple tensors that the map given by $f\otimes g \mapsto f \odot g$ is an isometric bimodule map with dense range inside $C(Gr(\sigma))$ with the norm $\|\cdot \|_w$. Indeed, let $\sum_{i=1}^{\ell}f_i \otimes g_i$ be a finite sum of simple tensors in $GNS(\sigma,w)$. Since 
$$
\la f_i \otimes g_i , f_j \otimes g_j \ra (x) = \sum_{k : x\in X_k} \overline{g_i(x)f_i(\sigma_k(x))}w_k(x)f_j(\sigma_k(x))g_j(x) = 
$$
$$
\sum_{k : x\in X_k} \overline{(f_i \odot g_i)(\sigma_k(x),x)}w_k(x)(f_j \odot g_j)(\sigma_k(x),x) = \la f_i \odot g_i , f_j \odot g_j \ra(x)
$$
we see that 
$$
\| \sum_{i=1}^{\ell}f_i \otimes g_i \|^2  = \sum_{i,j=1}^{\ell}\la f_i \otimes g_i , f_j \otimes g_j \ra (x) =
$$
$$
\sup_{x\in X} \sum_{i,j=1}^{\ell}\la f_i \odot g_i , f_j \odot g_j \ra (x) = \| \sum_{i=1}^{\ell}f_i \odot g_i \|^2
$$
Hence, $f\otimes g \mapsto f \odot g$ extends uniquely to the desired unitary.
\end{proof}

Our next goal is to compute the internal tensor iterates of $C(\sigma,w)$ for a WPS $(\sigma,w)$. As it turns out, the notation for paths in topological quivers fits this purpose very elegantly.

Recall the collection of paths in $\Qq(\sigma,w)$ of length $n$ given in Definition \ref{defi:paths},
$$
Gr(\sigma^n) := \{ \ \mu = \mu_n...\mu_1 \ | \ r(\mu_k) = s(\mu_{k+1}) \ \ \forall 1\leq k < n \ \}
$$
which can be alternatively identified with the closed set of orbits of length $n+1$ given by elements $(x_{n+1},x_n,...,x_1)$ in $X^{n+1}$ such that for all $1\leq m < n$ there is some $1\leq i \leq d$ such that $\sigma_i(x_m) = x_{m+1}$ and $x_m \in X_i$. 

Next, for functions $\xi, \eta \in C(Gr(\sigma^n))$ and $f,g\in C(X)$, left and right actions of $C(X)$ on $C(Gr(\sigma^n))$ are given by
$$
(f\cdot \xi \cdot g)(\mu) = f(r(\mu)) \xi(\mu) g(s(\mu))
$$
and the inner product by
$$
\la \xi, \eta \ra(x) = \sum_{s(\mu) = x} \overline{\xi(\mu)}w(\mu) \eta(\mu)
$$
where $w(\mu): = w(\mu_n)\cdot ... \cdot w(\mu_1)$ is the extended definition of the weights of edges to weights of paths. The above then yields the C*-correspondence structure associated to the topological quiver $\Qq(\sigma,w)^n$ on the space of $n$-paths as mentioned in the discussion at the beginning of Section 6 of \cite{Muhly-Tomforde}.

\begin{proposition} \label{proposition:iterate-computation}
Let $(\sigma,w)$ be a $d$-variable WPS. Then the map sending simple tensors $\xi_n \otimes ... \otimes \xi_1 \in C(Q(\sigma,w)^n)$ to the function $\xi_n \odot ... \odot \xi_1 : \mu_n...\mu_1 \mapsto \xi_n(\mu_n)\cdot ... \cdot \xi_1(\mu_1)$ in $C(Gr(\sigma^n))$ extends uniquely to an (Id-)unitary isomorphism between $C(\sigma,w)^{\otimes n}$ onto $C(Q(\sigma,w)^n)$ with the above C*-correspondence structure. In fact, the supremum norm on $C(Gr(\sigma^n))$ and the norm induced by the inner product on $C(Q(\sigma,w)^n)$ are equivalent.
\end{proposition}

\begin{proof}
The first part follows from Remark \ref{remark:tensor-iterate}. For the last part of the proposition, we note that for $x\in X$, the number of paths of length $n$ emanating from $x$ is at most $d^n$, and so, for every element $\xi \in C(Gr(\sigma^n))$, and an arbitrary path $\mu = \mu_n ... \mu_1$ of length $n$ emanating from $s(\mu)$ we have,
$$
\xi(\mu)w(\mu) \leq \sum_{s(\nu) = s(\mu)} |\xi(\nu)|^2w(\nu) \leq d^n \sup_{\nu \in Gr(\sigma^n)}\xi(\nu)w(\nu)
$$
By Proposition \ref{prop:weight-cont-bound}, we see that $\frac{1}{K^n} \leq w(\nu) = w(\nu_n)...w(\nu_1) \leq K^n$ for any $\nu = \nu_n ... \nu_1 \in Gr(\sigma^n)$ so that
$$
\frac{1}{K^n} \xi(\mu) \leq \sup_{x\in X} \sum_{s(\nu) = x} |\xi(\nu)|^2w(\nu) \leq d^n K^n \cdot \sup_{\nu \in Gr(\sigma^n)}\xi(\nu)
$$
Since $\mu$ was an arbitrary path, we see that 
$$
\frac{1}{K^n} \sup_{\mu \in Gr(\sigma^n)} \xi(\mu) \leq \sup_{x\in X} \sum_{s(\nu) = x} |\xi(\nu)|^2w(\nu) \leq d^n K^n \cdot \sup_{\nu \in Gr(\sigma^n)}\xi(\nu)
$$
and so the norm induced by the inner product on $C(Gr(\sigma^n))$ and the supremum norm are equivalent.
\end{proof}

\subsection{Weighted-orbit conjugacy} 
We now focus on the second conjugacy relation arising from our operator algebras, which we call weighted-orbit conjugacy. We give an example of two WPS which are not weighted-orbit conjugate, and an example of weight-orbit conjugate WPS which are not branch-transition conjugate. We conclude this subsection by providing a simple criterion for when graph, weighted-orbit and branch-transition conjugacy coincide.

\begin{defi}
Let $(\sigma,w)$ and $(\tau,u)$ be WPS on compact spaces $X$ and $Y$ respectively. We say that $(\sigma,w)$ and $(\tau,u)$ are \emph{weighted-orbit conjugate} with constant $C \geq 1$ if $\sigma$ and $\tau$ are graph conjugate via some homeomorphism $\gamma: X \rightarrow Y$ and there exists $H\in C(Gr(\sigma))$ such that for any $n\in \nn$ and any path $\mu = \mu_n ... \mu_1 \in Gr(\sigma^n)$ we have
$$
\frac{1}{C} \leq \Pi_{k=1}^n \Big[\frac{u^{\gamma}}{w}(\mu_k) H(\mu_k) \Big] \leq C
$$
\end{defi}

Intuitively, this means that multiplying by some continuous function $H$ on $Gr(\sigma)$ makes the gaps introduced by $\frac{u^{\gamma}}{w}$ uniformly bounded on paths of \emph{any} length. Note that when $C=1$, the above implies the continuity of $\frac{u^{\gamma}}{w}$ so that in this case $(\sigma,w)$ and $(\tau,u)$ are branch-transition conjugate.

\begin{example} \label{example:not-weighted-orbit-conj}
It turns out that the weighted multivariable systems of Example \ref{example:cpc-distinct-btc} are not even weighted-orbit conjugate, despite being graph conjugate. Indeed, for every $x \in [0,1]$ one can construct a path of length $n$ comprised of the same edge $e=(x,x)\in Gr(\sigma)$ at every step. In this case, if $H\in C(Gr(\sigma))$ and a (necessarily non-decreasing) homeomorphism $\gamma$ realize weighted-orbit conjugacy with constant $C$, we have that for $e=(x,x)$ such that $x>0$ and $n \in \nn$, 
$$
\frac{1}{C} \leq \Pi_{k=1}^n\Big[\frac{3}{2} \cdot H(e) \Big] \leq C
$$
so this forces $H(e) = \frac{3}{2}$. On the other hand, if $x=0$ we have
$$
\frac{1}{C} \leq \Pi_{k=1}^n H(e) \leq C
$$
which forces $H(e) = 1$, and $H$ cannot be continuous since $H(x,x)$ does not converge to $H(0,0)$ as $(x,x) \rightarrow (0,0)$ in $Gr(\sigma)$.
\end{example}

\begin{example} \label{example:different-invariants}
Weighted-orbit conjugacy does not imply branch-transition conjugacy, not even when the weights $w$ and $u$ are constant $1$. If we take $X=[0,1]$ and $\sigma_1(x) = \chi_{[0,\frac{1}{2}]}(x) + 2(1-x)\chi_{(\frac{1}{2},1]}$, $\sigma_2(x) = 0$, then the multivariable systems $\sigma = (\sigma_1,\sigma_2,\sigma_2)$ and $\tau = (\sigma_1,\sigma_1,\sigma_2)$, considered as WPS $(\sigma,w)$ and $(\tau,u)$ with constant weights $w = u = 1$, are not branch conjugate. Indeed, suppose $\sigma$ is graph conjugate to $\tau$ via $\gamma$. Then $\gamma(1) = 1$ as $1$ is the only branching point for both $\sigma$ and $\sigma'$, and so $\gamma$ must be non-decreasing. Hence, we have that
$$
\frac{u^{\gamma}}{w}(y,x) = \begin{cases} 
2 & \text{if} \ x < 1 \ \& \ y=\sigma_1(x) \\ 
\frac{1}{2} & \text{if} \ x < 1 \ \& \ y=\sigma_2(x) = 0 \\
1 & \text{if} \ x =1
\end{cases} 
$$ 
So we see that $\frac{u^{\gamma}}{w}$ is not continuous at the branching edge $e=(0,1)$, so that $(\sigma,w)$ and $(\sigma,u)$ are not branch transition conjugate. 

However, we show that $(\sigma,w)$ and $(\tau,u)$ are weighted-orbit conjugate via $\gamma = Id_{[0,1]}$. Note first that $Gr(\sigma) = Gr(\tau)$ so that $\sigma$ and $\tau$ are graph conjugate via $Id_X$. Next, we define the following continuous $H \in C(Gr(\sigma))$ by setting
$$
H(y,x) = \begin{cases} 
\frac{1}{2} & \text{if} \ x < \frac{1}{2} \ \& \ y=\sigma_1(x) \\ 
x & \text{if} \ \frac{1}{2} \leq x \leq 1 \ \& \ y=\sigma_1(x) \\ 
2 & \text{if} \ x < \frac{1}{2} \ \& \ y=\sigma_2(x) = 0 \\
-2x + 3 & \text{if} \ \frac{1}{2} \leq x \leq 1 \ \& \ y= \sigma_2(x)= 0 \\
1 & \text{if} \ x =1
\end{cases}
$$
and by our definition of $H$ for $e\in Gr(\sigma)$ with $s(e)\leq \frac{1}{2}$ we have $H(e)\frac{u}{w}(e) = 1$. The important thing to note here is that every path beginning at some $x\in [0,1]$ must be comprised, from the third edge on, by edges $e$ with $r(e),s(e) \in \{0,\frac{1}{2}\}$. Indeed, if $\mu = \mu_n...\mu_1$ is a path of length $|\mu|\geq 3$, suppose that $s(\mu_1) = x$, then $r(\mu_1) \in [0,\frac{1}{2}]$, and this forces $r(\mu_2) \in \{0,\frac{1}{2}\}$. 

Hence, we see that for any $n\in \nn$ and any path $\mu= \mu_n...\mu_1$ we have
$$
\Pi_{k=1}^n \frac{u(\mu_k)H(\mu_k)}{w(\mu_k)} = \Pi_{k=1}^2\frac{u(\mu_k)H(\mu_k)}{w(\mu_k)}
$$
Since both $H$ and $\frac{w}{u}$ have values only in the interval $[\frac{1}{2}, 2]$, we see that 
$$
\Big(\frac{1}{2}\Big)^4 \leq \Pi_{k=1}^n \frac{u(\mu_k)H(\mu_k)}{w(\mu_k)} \leq 2^4
$$ 
and so $(\sigma,w)$ and $(\sigma,u)$ are weighted-orbit conjugate via $Id_{[0,1]}$.
\end{example}

It is easy to see that the conjugacy relations we have defined between two WPS have a natural hierarchy. By definition, if $(\sigma,w)$ and $(\tau,u)$ are two WPS over $X$ and $Y$ respectively, then each condition below implies the one after it:
\begin{enumerate}
\item
$(\sigma,w)$ and $(\tau,u)$ are conjugate.
\item
$(\sigma,w)$ and $(\tau,u)$ are branch-transition conjugate.
\item
$(\sigma,w)$ and $(\tau,u)$ are weighted-orbit conjugate.
\item
$(\sigma,w)$ and $(\tau,u)$ are graph conjugate.
\end{enumerate}
As we have seen, the different conjugacy relations are distinct, but in some subclasses, it is possible to identify some of them. 
\begin{enumerate}
\item
For partially defined continuous functions as in Subsection \ref{subsec:pdcf}, graph conjugacy implies conjugacy.
\item
For non-negative matrices as in Subsection \ref{subsec:nnm}, graph conjugacy implies branch-transition conjugacy.
\end{enumerate}

In general we have the following in the case when there are no branching points, which tells us that information on the weights can only be detected if the WPS have branching points.

\begin{corollary} \label{cor:no-branching-graph}
Let $(\sigma,w)$ and $(\tau,u)$ be WPS over compact $X$ and $Y$ respectively. Suppose either $\sigma$ or $\tau$ have no branching points. Then $\sigma$ and $\tau$ are graph conjugate if and only if $(\sigma,w)$ and $(\tau,u)$ are branch-transition conjugate.
\end{corollary}

\begin{proof}
If $\sigma$ and $\tau$ are graph conjugate and either $\sigma$ or $\tau$ have no branching points, then both have no branching points by Remark \ref{rem:graph-conj-preserves-branching-edges}. Hence by Proposition \ref{prop:weight-cont-bound} we see that both $w$ and $u$ are continuous, and so for a homeomorphism $\gamma : X \rightarrow Y$ such that $Gr(\sigma) = Gr(\tau^{\gamma})$ we have that $\frac{u^{\gamma}}{w}$ is continuous, and $(\sigma,w)$ is branch-transition conjugate to $(\tau,u)$.
\end{proof}

\subsection{Conjugacy relations in terms of C*-correspondences}

We now characterize when two WPS are branch-transition / weighted-orbit conjugate in terms of unitary isomorphism / similarity of the associated C*-correspondences respectively. 

For $\gamma$ implementing branch-transition / weighted-orbit conjugacy and $\rho$ implementing unitary isomorphism / similarity, we may often assume without loss of generality that $\gamma = Id_X$ and / or that $\rho = Id_{C(X)}$. 

Indeed, if $V: C(\sigma,w)\rightarrow C(\tau,u)$ is a $\rho$-bimodule map, with $\gamma :Y \rightarrow X$ the homeomorphism such that $\rho(f) = f \circ \gamma^{-1}$, we may define a $\rho$-unitary $\widetilde{\rho} : C(\tau^{\gamma},u^{\gamma}) \rightarrow C(\tau,u)$ given by $\widetilde{\rho}(\xi)(y,v) = \xi(\gamma^{-1}(y),\gamma^{-1}(v))$. $\widetilde{\rho}$ then satisfies $\widetilde{\rho^{-1}} = \widetilde{\rho}^{-1}$ where $\widetilde{\rho^{-1}}$ is a $\rho^{-1}$-unitary. Hence, by composing we get an $Id$-bimodule map $\widetilde{\rho}^{-1} \circ V :C(\sigma,w) \rightarrow C(\tau^{\gamma},u^{\gamma})$, and $V$ is a $\rho$ - similarity / unitary if and only if $\widetilde{\rho}^{-1} \circ V$ is an $Id$ - similarity /  unitary respectively.

Further, on the conjugacy side, note that $(\sigma,w)$ and $(\tau,u)$ are graph / weighted-orbit / branch-transition conjugate via $\gamma$ if and only if $(\sigma,w)$ and $(\tau^{\gamma},u^{\gamma})$ are graph / weighted-orbit / branch transition conjugate via $Id_X$ respectively.

\begin{proposition} \label{prop:weighted-orbit-char}
Let $(\sigma,w)$ and $(\tau,u)$ be WPS on compact spaces $X$ and $Y$ respectively. Suppose that $\gamma: X \rightarrow Y$ is a homeomorphism and $\rho: C(X) \rightarrow C(Y)$ is the *-isomorphism given by $\rho(f) =f \circ \gamma^{-1}$.

\begin{enumerate}
\item
If $(\sigma,w)$ and $(\tau,u)$ are weighted-orbit conjugate with $C\geq 1$ via $\gamma$, then there exists a $\rho$-similarity $V: C(\sigma,w) \rightarrow C(\tau,u)$ with 
$$
\sup_n \max \{ \|V^{\otimes n} \|^2, \|(V^{-1})^{\otimes n} \|^2 \} \leq C
$$
\item
If $V: C(\sigma,w) \rightarrow C(\tau,u)$ is a $\rho$-similarity, then $(\sigma,w)$ and $(\tau,u)$ are weighted-orbit conjugate via $\gamma$ and constant
$$
C = \sup_n \max \{ \|V^{\otimes n} \|^2, \|(V^{-1})^{\otimes n} \|^2 \}
$$

\end{enumerate}
\end{proposition}

\begin{proof}
We first show (1). Assume without loss of generality that $\gamma = Id_X$, so that $Gr(\sigma) = Gr(\tau)$. Let $H\in C(Gr(\sigma))$ be such that for any path $\mu = \mu_n...\mu_1$ we have
$$
\frac{1}{C} \leq \Pi_{k=1}^n \frac{u}{w}(\mu_k) H(\mu_k) \leq C
$$
We define $V: C(\sigma,w) \rightarrow C(\tau,u)$ by setting $V(\xi)(e) = \xi(e) \sqrt{H(e)}$. It is easily seen that $V$ is a $C(X)$-bimodule map, and we show that $V$ is an $Id$-isomorphism. Indeed, for $\xi \in C(Gr(\sigma))$ we have
$$
\|V(\xi) \|^2 = \sup_{x\in X} \sum_{s(e) = x}|\xi(e)|^2 u(e) H(e) \leq 
C \cdot \sup_{x\in X}\sum_{s(e) = x}|\xi(e)|^2 w(e) = C \|\xi \|^2
$$
and the symmetric argument shows that $\|\xi \| \leq C \|V(\xi)\|$. Hence $V : C(\sigma,w) \rightarrow C(\tau,u)$ is an $Id$-isomorphism. To show that it is an $Id$-similarity, we repeat the above for the tensor iterates which are identified with $C(Gr(\sigma^n))$ for $n\in \nn$ by Proposition \ref{proposition:iterate-computation}. Indeed, fix $n\in \nn$, and $\xi \in C(Gr(\sigma^n))$. By Proposition \ref{proposition:iterate-computation} and the definition of $V$, we must have that $V^{\otimes n}(\xi)(\mu_n...\mu_1) = \xi(\mu_n...\mu_1)\Pi_{k=1}^n\sqrt{H(\mu_k)}$. Thus, we compute,
$$
\|V^{\otimes n}(\xi) \|^2 = \sup_{x\in X} \sum_{s(\mu_n...\mu_1) = x} |\xi(\mu_n...\mu_1)|^2 \Pi_{k=1}^n H(\mu_k) u(\mu_k) \leq 
$$
$$
C \cdot \sup_{x\in X} \sum_{s(\mu_n...\mu_1) = x} |\xi(\mu_n...\mu_1)|^2 \Pi_{k=1}^n w(\mu_k) = C \|\xi \|^2
$$
So that $V$ is tensor-power bounded by $\sqrt{C}$ and the symmetric argument shows that $V^{-1}$ is also tensor power bounded by $\sqrt{C}$.

We now show (2). Without loss of generality we assume that  $\rho= Id_{C(X)}$ (so that we need $\gamma = Id_X$). Denote by $\zeta = V(1\odot 1) \in C(\tau,u)$. For any $f,g\in C(X)$ we have $f\cdot \zeta \cdot g = V(f\odot g)$ and then
$$
\sup_{x\in X} \sum_{s(e) = x}|f(r(e))|^2 |\zeta(e)|^2 |g(s(e))|^2 u(e) = \|V(f\odot g) \| \leq 
$$
$$
\|V\| \|f\odot g\| = \|V \| \sup_{x\in X} \sum_{s(e) = x}|f(r(e))|^2 |g(s(e))|^2 w(e)
$$
so we see that for $(y,x) \in Gr(\sigma)$, by taking infimum over $f,g :X \rightarrow [0,1]$ with $f(y) = 1$ and $g(x) = 1$ which vanish outside arbitrarily small neighborhoods of $y$ and $x$ respectively, we have that $(y,x) \in Gr(\tau)$, for otherwise the right hand side would vanish while the left hand side would not.
The symmetric argument then shows that $Gr(\sigma) = Gr(\tau)$, and $\sigma$ and $\tau$ are graph conjugate via $Id_X$.

Next, by Proposition \ref{proposition:GNS-computation}, convergence in $C(\sigma,w)$ is equivalent to uniform convergence on $C(Gr(\sigma))$, and since $\zeta(e) \cdot (f \odot g)(e) = (f\cdot \zeta \cdot g)(e) = V(f\odot g)(e)$ for every $e\in Gr(\sigma)$, we then must have that $V(\xi)(e) = \zeta(e) \cdot \xi(e)$ for every $\xi \in C(Gr(\sigma))$ and $e\in Gr(\sigma)$.

Next, since for every $\xi_k \in C(Gr(\sigma)) = C(Gr(\tau))$ for $ 1 \leq k \leq n$ we have that
$$
\|\xi_n \otimes ... \otimes \xi_1 \|^2 \leq 
\|(V^{-1})^{\otimes n} \|^2 \|V^{\otimes n}(\xi_n \otimes ... \otimes \xi_1) \|^2
$$
and
$$
\|V^{\otimes n}(\xi_n \otimes ... \otimes \xi_1) \|^2 \leq 
\|V^{\otimes n} \|^2 \|\xi_n \otimes ... \otimes \xi_1 \|^2
$$
We obtain that
$$
\sup_{x\in X} \sum_{s(\mu_n ... \mu_1) = x} \Pi_{k=1}^n |\xi_k(\mu_k)|^2 w(\mu_k) \leq 
$$
$$
\|(V^{-1})^{\otimes n} \|^2 \sup_{x\in X} \sum_{s(\mu_n ... \mu_1) = x} \Pi_{k=1}^n |\xi_k(\mu_k)|^2 |\zeta(\mu_k)|^2 u(\mu_k)
$$
and
$$
\sup_{x\in X} \sum_{s(\mu_n ... \mu_1) = x} \Pi_{k=1}^n |\xi_k(\mu_k)|^2 |\zeta(\mu_k)|^2 u(\mu_k) \leq 
$$
$$
\|V^{\otimes n} \|^2 \sup_{x\in X} \sum_{s(\mu_n ... \mu_1) = x} \Pi_{k=1}^n |\xi_k(\mu_k)|^2 w(\mu_k)
$$
If we \emph{fix} a path $\nu= \nu_n...\nu_1$, we can take infimum over functions $\xi_k$ that vanish outside arbitrarily small neighborhoods of $\nu_k$ for each $k$ and are equal to $1$ at $\nu_k$, to get
$$
\Pi_{k=1}^n w(\nu_k) \leq \|(V^{-1})^{\otimes n} \|^2 \Pi_{k=1}^n |\zeta(\nu_k)|^2 u(\nu_k)
$$
and
$$
\Pi_{k=1}^n |\zeta(\nu_k)|^2 u(\nu_k) \leq \|V^{\otimes n} \|^2 \Pi_{k=1}^n w(\nu_k)
$$
so that with 
$$
C = \max \{ \sup_n\|V^{\otimes n} \|^2 , \sup_n \|(V^{-1})^{\otimes n} \|^2 \}
$$
which by our assumptions is finite, we get
$$
\frac{1}{C} \leq  \frac{\Pi_{k=1}^n H(\nu_k) u(\nu_k)}{\Pi_{k=1}^n w(\nu_k)} = \Pi_{k=1}^n \frac{u}{w}(\nu_k) H(\nu_k) \leq C
$$
Where $H = |\zeta|^2 \in C(Gr(\sigma))$, as required.
\end{proof}

As a corollary to the above, we obtain a characterization for branch-transition conjugacy.

\begin{corollary} \label{cor:branch-transition-char}
Let $(\sigma,w)$ and $(\tau,u)$ be WPS on compact spaces $X$ and $Y$ respectively. Suppose that $\gamma: X \rightarrow Y$ is a homeomorphism and $\rho: C(X) \rightarrow C(Y)$ is the *-isomorphism given by $\rho(f) =f \circ \gamma^{-1}$.

\begin{enumerate}
\item
If $(\sigma,w)$ and $(\tau,u)$ are branch-transition conjugate via $\gamma$, then there exists a $\rho$-unitary $U: C(\sigma,w) \rightarrow C(\tau,u)$.
\item
If $U: C(\sigma,w) \rightarrow C(\tau,u)$ is a $\rho$-unitary, then $(\sigma,w)$ and $(\tau,u)$ are branch-transition conjugate via $\gamma$.
\end{enumerate}
\end{corollary}

\begin{proof}
To show (1), we use Corollary \ref{corollary:radon-nikodym-continuity} to see that $H = \frac{w}{u^{\gamma}}$ is continuous, and realizes weighted-orbit conjugacy with $C=1$, so that the $\rho$-similarity $U$ arising from Proposition \ref{prop:weighted-orbit-char} satisfies $\|U\| , \|U^{-1}\| \leq 1$ and is hence a $\rho$-unitary.

To show (2), without loss of generality we assume that $\rho= Id_{C(X)}$ (So that we need $\gamma = Id_X$). Denote by $\zeta = V(1\odot 1) \in C(\tau,u)$. For any $f,g\in C(X)$ we have $f\cdot \zeta \cdot g = U(f\odot g)$ and then
$$
\sup_{x\in X} \sum_{s(e) = x}|f(r(e))|^2 |\zeta(e)|^2 |g(s(e))|^2 u(e) = \|U(f\odot g) \| = 
$$
$$
\|f\odot g\| = \sup_{x\in X} \sum_{s(e) = x}|f(r(e))|^2 |g(s(e))|^2 w(e)
$$
so we see that for $e= (y,x) \in Gr(\sigma)$, by taking infimum over $f,g :X \rightarrow [0,1]$ with $f(y) = 1$ and $g(x) = 1$ which vanish outside arbitrarily small neighborhoods of $y$ and $x$ respectively, we obtain $|\zeta(e)|^2u(e) = w(e)$ so that $e\in Gr(\sigma)$ if and only if $e\in Gr(\tau)$ and $\sigma$ and $\tau$ are graph conjugate via $Id$. Moreover, since $\frac{u}{w} = \frac{1}{|\zeta|^2}$ is a continuous function on $C(Gr(\sigma))$, it must be continuous on each branching edge in particular, and hence $(\sigma,w)$ and $(\tau,u)$ are branch-transition conjugate.
\end{proof}

\begin{example} \label{ex:non-adjointable}
As a consequence of Proposition \ref{prop:weighted-orbit-char} and Corollary \ref{cor:branch-transition-char} we see from Example \ref{example:different-invariants} that there are WPS which have similar C*-correspondences that can not be unitarily isomorphic. In particular, by Remark \ref{remark:adjointable-isomorphism-implies-unitary} we see that between the two correspondences arising from the weighted multivariable systems of Example \ref{example:different-invariants}, no $\rho$-isomorphism can be $\rho$-adjointable.
\end{example}

\begin{remark} \label{remark:graph-conj-char}
Using the theory we have developed so far, and the first part of Corollary \ref{corollary:radon-nikodym-continuity}, one can show that for two WPS $(\sigma,w)$ and $(\tau,u)$ over compact spaces $X$ and $Y$ respectively, we have $\sigma$ and $\tau$ graph conjugate via $\gamma$ if and only if $C(\sigma,w)$ and $C(\tau,u)$ are $\rho$-isomorphic. Hence, $\rho$-isomorphism does not detect any information regarding the weights of the WPS, and only detects the graphs of the systems. 
\end{remark}

\section{Tensor algebras} \label{sec:tensor-algebras}

In general, one can associate to every C*-correspondence $E$ over $\Aa$ a non-self-adjoint norm closed operator algebra $\Tt_+(E)$ called the \emph{tensor algebra} of $E$. We commence a discussion of this general theory, adapting and building on Section 5 of \cite{Muhly-Solel-Morita} and ideas from Section 6 of \cite{StochMat}.

\subsection{Construction}
To construct the tensor algebra associated to a C*-correspondence $E$ over $\Aa$, we first construct the Fock direct sum C*-correspondence

\begin{equation*}
\Ff_E := \bigoplus_{n\in \mathbb{N}}E^{\otimes n} 
\end{equation*} 

The $E$-shifts are the operators $S_{\xi} \in \Ll(\Ff_E)$ for $\xi \in E$, uniquely determined by defining them on direct summands via the equation $S_{\xi}(\eta) = \xi \otimes \eta $, for $m \in \mathbb{N}$, $\eta \in E^{\otimes m}$.

\begin{defi}
The tensor algebra $\Tt_+(E)$ is the norm-closed subalgebra of $\Ll(\Ff_E)$ generated by all $E$-shifts and $\Aa$,
$$ 
\Tt_+(E) : = \overline{Alg}(\Aa \cup \{ \ S_{\xi} \ | \ \xi \in E \ \})
$$

The Toeplitz algebra $\Tt(E)$ is the C*-subalgebra of $\Ll(\Ff_E)$ generated by all $E$-shifts and $\Aa$. That is,

$$\Tt(E) = C^*(\Tt_+(E)) = C^*(\Aa \cup \{ \ S_{\xi} \ | \ \xi \in E \ \})$$
\end{defi}

The algebra $\Ll(\Ff_E)$ admits a natural action $\alpha$ of the unit circle $\mathbb{T}\subset \mathbb{C}$, called the gauge action, defined by $\alpha_{\lambda}(T) = W_{\lambda}TW_{\lambda}^*$ for all $\lambda \in \mathbb{T}$ where $W_{\lambda}: \Ff_E \rightarrow \Ff_E$ is the unitary defined by 
$$W_{\lambda}(\oplus_{n\in \mathbb{N}} \xi_n) = \oplus_{n\in \mathbb{N}} \lambda^n \xi_n$$

Since $\alpha_{\lambda}(S_{\xi}) = S_{\lambda \xi}$ and $\alpha_{\lambda}(a) = a$ for $a \in \Aa$ and $\xi \in E$, it follows that both the Toeplitz algebra and tensor algebra are $\alpha$-invariant closed subalgebras, so the circle action can be restricted to a completely isometric circle action on each of them. One then shows that for every $S\in \Tt(E)$, the function $f(\lambda) = \alpha_{\lambda}(S)$ is norm continuous, and this enables the definition of a conditional expectation $\Phi$ given by
$$ \Phi(S) =\int_{\mathbb{T}}\alpha_{\lambda}(S)d \lambda $$
Where $d \lambda$ is the normalized Haar measure on $\mathbb{T}$.

Let $\{ k_n \}_{n=1}^{\infty}$ denote Fejer's kernel function defined for $\lambda \in \mathbb{T}$ by 
$$k_n(\lambda) = \sum_{j=-n}^n\big(1 - \frac{|j|}{n+1}\big) \lambda^j$$
Note that for $S\in \Tt(E)$, the existence of the canonical conditional expectation $\Phi$ permits the definition of Fourier coefficients for an element $S \in \Tt(E)$ by
$$\Phi_n(S) = \int_{\mathbb{T}} \alpha_{\lambda}(S)\lambda^{-n}d \lambda$$
Then define the Cesaro sums,
$$ \sigma_n(S) := \sum_{j = -n}^n\big(1 - \frac{|j|}{n+1}\big) \Phi_j(S) =
\int_{\mathbb{T}} \sum_{j = -n}^n\big(1 - \frac{|j|}{n+1}\big) \alpha_{\lambda}(S)\lambda^{-j}d \lambda
= \int_{\mathbb{T}} \alpha_{\lambda}(S)k_n(\lambda)d \lambda
$$

Every tensor algebra is graded by the spaces 
$$\Tt_+(E)_n  =  
\Phi_n(\Tt_+(E)) =
\overline{\Span} \{\  S_{\xi_1}\cdot ... \cdot S_{\xi_n} \ \ | \ \ \xi_1,...,\xi_n \in E \ \}$$
We denote by $S^{(n)}_{\xi} \in \Ll(\Ff_E)$, for $\xi \in E^{\otimes n}$, the operator determined uniquely by $S^{(n)}_{\xi}(\eta) = \xi\otimes \eta$ for $m\in \nn$ and $\eta \in E^{\otimes m}$.

The following is a folklore result for tensor algebras that relates the above notions. We refer the reader to Proposition 6.2 in \cite{StochMat} for a proof of this result in the case of subproduct systems over W*-algebras, which is easily adapted to our context.

\begin{proposition} \label{proposition:fourier-series}
Let $E$ be a C*-correspondence over $\Aa$. For every $n \in \mathbb{N}$ 
we have that $E^{\otimes n}$ is isometrically isomorphic 
as a Banach $\Aa$-bimodule to $\tensor(E)_n$
via the map determined uniquely by $\xi \mapsto S^{(n)}_{\xi}$.

Therefore, every element $T\in \tensor(E)$ 
has a unique representation as an infinite series
$T = \sum_{n = 0}^{\infty}S^{(n)}_{\xi_n}$
where $\xi_n \in E^{\otimes n}$ satisfies $\Phi_n(T) = S^{(n)}_{\xi_n}$
(called its Fourier series representation for short),  and the series 
converges Cesaro to $T$ in norm: 
if 
$\sigma_N(T) = \sum_{n=0}^{N}\Big(1 - \frac{n}{N+1} \Big)S^{(n)}_{\xi_n}$,
then we have that
$\lim_{N\to \infty} \| \sigma_N(T) - T\| = 0$.
Furthermore, if $T, T' \in \tensor(E)$ have Fourier series representations
$T = \sum_{i=0}^{\infty}S^{(i)}_{\xi_i}$ and $T' = \sum_{i=0}^{\infty}S^{(i)}_{\eta_i}$, then 
$$
TT' = \sum_{n=0}^{\infty}S^{(n)}_{\zeta},
\quad \text{where} \quad
\zeta= \sum^{n}_{k=0}\xi_k \otimes \eta_{n-k}.
$$
\end{proposition}

\subsection{Graded isomorphisms}

We focus now on the analysis of graded isomorphisms, which are isomorphisms $\varphi: \tensor(E) \rightarrow \tensor(F)$ that satisfy $\varphi(\tensor(E)_n) = \tensor(F)_n$ for all $n\in \nn$.

From now on, we restrict our discussion to C*-correspondences over \emph{commutative} C*-algebras. This will allow us to obtain a relationship between isomorphisms of the C*-correspondences, and graded isomorphisms of the tensor algebras.

For a C*-correspondence $E$, let $\Psi_E : E \rightarrow \tensor(E)_1$ be the isometric Banach bimodule isomorphism given by $\Psi_E(\xi) = S_{\xi}^{(1)}$. In the following theorem it is important that we do not require $\rho$-similarities to be adjointable in item (2) of Definition \ref{defi:isomorphisms-C-corresp}.

\begin{theorem} \label{theorem:bounded-graded}
Let $E$ and $F$ be be C*-correspondences over \emph{commutative} C*-algebras $\Aa$ and $\Bb$ respectively. Then,
\begin{enumerate}
\item
If $V: E \rightarrow F$ is a $\rho$-similarity for some *-isomorphism $\rho$ between $\Aa$ and $\Bb$, then there exists a graded completely bounded isomorphism $\Ad_V:\Tt_+(E) \rightarrow \Tt_+(F)$ such that $\Ad_V |_{\Aa} = \rho$ with 
$$
\max\{\|\Ad_V \|_{cb}, \|\Ad_V^{-1} \|_{cb} \} \leq \sup_{n\in \nn}\|V^{\otimes n}\| \cdot \sup_{n\in \nn}\|(V^{-1})^{\otimes n}\|
$$
\item
If $\varphi:\Tt_+(E) \rightarrow \Tt_+(F)$ is a bounded graded isomorphism, then $\rho_{\varphi}: = \varphi |_{\Aa} : \Aa \rightarrow \Bb$ is a *-isomorphism and $V_{\varphi}: E \rightarrow F$ uniquely determined by $S_{V_{\varphi}(\xi)} =\varphi(S_{\xi})$ for $\xi \in E$ yields a $\rho_{\varphi}$-similarity satisfying
$$
\sup_{n\in \nn}\|(V_{\varphi})^{\otimes n} \| \leq \|\varphi \| \ \ and \ \ 
\sup_{n\in \nn}\|(V_{\varphi}^{-1})^{\otimes n} \| \leq \|\varphi^{-1} \|
$$
\end{enumerate}
Moreover, the operations (1) and (2) are inverses of each other in the sense that $\varphi = \Ad_{V_{\varphi}}$ and $V = V_{Ad_V}$, and in particular every bounded graded isomorphism $\varphi$ is completely bounded with $\|\varphi \|_{cb} \leq \|\varphi \| \cdot \|\varphi^{-1} \|$.
\end{theorem}

\begin{proof}
(1) Suppose $V:E \rightarrow F$ is a $\rho$-similarity. Define a $\rho$-correspondence morphism $W_V$ from $\Ff_E$ to $\Ff_F$ by $W_V = \oplus_{n=0}^{\infty}V^{\otimes n}$ which is well defined since $\sup_{n\in \nn}\|V^{\otimes n}\| < \infty$. Furthermore, we have that $W_V$ is invertible with $W_V^{-1} = W_{V^{-1}}$ (which is also a well-defined $\rho^{-1}$-correspondence morphism since $\sup_{n\in \nn}\|(V^{-1})^{\otimes n}\| < \infty$) and 
$$
\|W_V \|\cdot  \|W_V^{-1} \| \leq \sup_{n\in \nn}\|V^{\otimes n}\| \cdot \sup_{n\in \nn}\|(V^{-1})^{\otimes n}\|
$$
We define $Ad_V: \tensor(E) \rightarrow \tensor(F)$ by setting $Ad_V(T) = W_VTW_V^{-1}$ which then satisfies 
$$\max \{ \|Ad_V\|_{cb}, \|Ad_V^{-1}\|_{cb} \} \leq \sup_{n\in \nn}\|V^{\otimes n}\| \cdot \sup_{n\in \nn}\|(V^{-1})^{\otimes n}\|$$
(2) Now suppose that $\varphi : \tensor(E) \rightarrow \tensor(F)$ is a bounded graded isomorphism. Note that $\rho_{\varphi}$ is a *-isomorphism since $\Aa$ and $\Bb$ are assumed commutative. We define the map $V_{\varphi} : E \rightarrow F$ by $V_{\varphi} = \Psi_F^{-1} \varphi \Psi_E$ which is a $\rho$-correspondence map by virtue of gradedness of $\varphi$ and the fact that $\rho_{\varphi}$ is a *-isomorphism. Then it is easily verified that for all $n\in \nn$ we have $(V_{\varphi})^{\otimes n} = (\Psi_F^{-1})^{\otimes n} \circ \varphi \circ (\Psi_E)^{\otimes n}$ so that $\|(V_{\varphi})^{\otimes n}\| \leq \| \varphi \|$ and $V_{\varphi}$ is tensor-power bounded with $\sup_n \|(V_{\varphi})^{\otimes n}\| \leq \| \varphi \|$. One then easily shows that $V_{\varphi}^{-1} = V_{\varphi^{-1}} = \Psi_E^{-1} \varphi^{-1} \Psi_F$ is also a $\rho^{-1}$-correspondence map which is similarly tensor power-bounded with $\sup_n \|(V_{\varphi}^{-1})^{\otimes n} \| \leq \|\varphi^{-1}\|$, as required.
\end{proof}

We then get as an easy corollary, the corresponding theorem for the isometric case.

\begin{theorem} \label{theorem:isometric-graded}
Let $E$ and $F$ be C*-correspondences over commutative C*-algebras $\Aa$ and $\Bb$ respectively.
\begin{enumerate}
\item
If $U: E \rightarrow F$ is a $\rho$-unitary for some *-isomorphism $\rho$ between $\Aa$ and $\Bb$, then there exists a graded completely isometric isomorphism $\Ad_U: \Tt_+(E) \rightarrow \Tt_+(F)$ with $\Ad_U | _{\Aa} = \rho$.
\item
If $\varphi: \Tt_+(E) \rightarrow \Tt_+(F)$ is a graded isometric isomorphism, then $\rho_{\varphi}$ is a *-isomorphism and there exists a $\rho_{\varphi}$-unitary $U_{\varphi}: E\rightarrow F$ with $\varphi | _{\Aa} = \rho_{\varphi}$.
\end{enumerate}
Moreover, the operations (1) and (2) are inverses of each other in the sense that $\varphi = \Ad_{V_{\varphi}}$ and $V = V_{Ad_V}$, and in particular every isometric graded isomorphism $\varphi$ must be completely isometric.
\end{theorem}

\subsection{Base-detection and semi-gradedness}
We will now consider two special classes of isomorphisms which will provide a convenient framework for addressing isomorphism problems. When our isomorphisms fit into these classes, we can often use this to extract information more readily from the isomorphism.

\begin{notation}
If $E$ and $F$ are C*-correspondences over C*-algebra $\Aa$ and $\Bb$ respectively and $\varphi : \Tt_+(E) \rightarrow \Tt_+(F)$ is an algebraic isomorphism, we denote by $\rho_{\varphi} : = \Phi_0 \circ \varphi |_{\Aa}$ which is a homomorphism between $\Aa$ and $\Bb$.
\end{notation}

\begin{defi}
Let $E$ and $F$ be C*-correspondences over C*-algebras $\Aa$ and $\Bb$ respectively. We say that an algebraic isomorphism $\varphi :\Tt_+(E) \rightarrow \Tt_+(F)$ is \emph{base-detecting} if $\rho_{\varphi} : \Aa \rightarrow \Bb$ is a *-isomorphism and $\rho_{\varphi}^{-1} = \rho_{\varphi^{-1}}$.
\end{defi}

Base detection is usually the first thing we check for, since it usually implies that the base algebras can be detected from the isomorphism.

We note that for a graded isomorphism $\varphi$ as considered in Theorem \ref{theorem:bounded-graded}, $\rho_{\varphi}$ is automatically an isomorphism, and since they were between commutative C*-algebras, $\rho_{\varphi}$ had to be a *-isomorphisms. This means that graded isomorphisms are always base-detecting.

Isometric isomorphisms are also automatically base detecting. Indeed, let $E$ and $F$ be C*-correspondences over C*-algebras $\Aa$ and $\Bb$ and let $\varphi: \Tt_+(E) \rightarrow \Tt_+(F)$ be an isometric isomorphism. Since $\Tt_+(F) \subset \Tt(F)$, we can regard $\varphi$ as a map into the Toeplitz C*-algebra. Thus, $\varphi | _{\Aa} : \Aa \rightarrow \Tt(F)$ is an isometric homomorphism, and is hence necessarily positive and preserves the involution from $\Aa$ to $\Tt(F)$. Thus, $\varphi(\Aa) = \varphi(\Aa)^* \subset \Tt_+(F)^* \subset \Tt(F)$, and we must have that $\varphi(\Aa) \subset \Tt_+(F) \cap \Tt_+(F)^* = \Bb$. Thus we have in fact that $\varphi(\Aa) \subset \Bb$, and the symmetric argument shows that $\varphi^{-1}(\Bb) \subset \Aa$, and so $\rho_{\varphi^{-1}}$ is the inverse of $\rho_{\varphi}$, and $\varphi$ is base-detecting.

We now try to relax the assumption of gradedness of an isomorphism while trying to maintain base-detection.

The following concept of semi-gradedness appeared in the work of Muhly and Solel in section 5 of \cite{Muhly-Solel-Morita} where they resolve the isometric isomorphism problem for tensor algebras arising from aperiodic C*-correspondences, and was also used in \cite{StochMat} to provide classification for tensor algebras arising from stochastic matrices, in terms of the matrices.

\begin{defi}
Let $E$ and $F$ be C*-correspondences over C*-algebras $\Aa$ and $\Bb$ respectively, and suppose $\varphi: \Tt_+(E) \rightarrow \Tt_+(F)$ is an algebraic isomorphism. We say that $\varphi$ is \emph{semi-graded} if $\varphi(\Ker \Phi_0) = \Ker \Phi_0$.
\end{defi}

We now provide an analogous proof to the one of Proposition 6.15 in \cite{StochMat}, that semi-gradedness implies base-detection.

\begin{proposition} \label{prop:base-detect-semi-graded}
Let $E$ and $F$ be C*-correspondences over commutative C*-algebras $\Aa$ and $\Bb$ respectively, and $\varphi :\Tt_+(E) \rightarrow \Tt_+(F)$ is a semi-graded bounded isomorphism. Then $\varphi$ is automatically base-detecting.
\end{proposition}

\begin{proof}
Let $\Phi_0^E$ and $\Phi_0^F$ denote the conditional expectations on $\Tt_+(E)$ and $\Tt_+(F)$ respectively. As $\varphi$ is semi-graded, for any $T\in \Tt_+(E)$ we have,
$$
\Phi_0^F \varphi (T) = \Phi_0^F \varphi \Phi_0^E(T)
$$
Hence, we must have that $\rho_{\phi} = \Phi_0^F \varphi |_{\Aa}$ is surjective. The same argument then works for $\varphi^{-1}$, and we have for every $a\in \Aa$ that
$$
\rho_{\varphi^{-1}} \circ \rho (a) = \Phi_0^E \varphi^{-1} \Phi_0^F \varphi (a) = \Phi_0^E \varphi^{-1}\varphi(a) = a
$$
Thus, we see that $\rho_{\varphi^{-1}} = (\rho_{\varphi})^{-1}$. As $\Aa$ and $\Bb$ are commutative, $\rho_{\varphi}$ and $\rho_{\varphi^{-1}}$ must both be contractive, and hence *-preserving, so that $\varphi$ is base-detecting.
\end{proof}

\begin{defi}
Let $E$ be a C*-correspondence over $\Aa$. The minimal degree of an element $0 \neq T\in \tensor(E)$, denoted $\md(T)$ is the smallest $n\in \nn$ with $\Phi_n(T) \neq 0$.
\end{defi}

We will need the following criterion for semi-gradedness of \emph{bounded} isomorphisms.

\begin{proposition}[Criterion for semi-gradedness] \label{proposition:criterion-semi-graded}
Let $E$ and $F$ be C*-correspondences over $\Aa$ and $\Bb$ respectively, and let $\varphi : \tensor(E) \rightarrow \tensor(F)$ be a \emph{bounded} base-detecting isomorphism. The following are equivalent:
\begin{enumerate}
\item
$\md(\varphi(T)) = \md(T)$ for all $T\in \tensor(E)$
\item
$\varphi$ is semi-graded
\item
$\md(\varphi(S^{(1)}_{\xi})) \geq 1$ for every $\xi \in E$.
\end{enumerate}

\end{proposition}

\begin{proof}
It is clear that (1) implies (2) which implies (3). 

We show that (3) implies (1). We first note that for $\eta \in E^{\otimes n}$ we have that $\md(\varphi(S_{\eta}^{(n)})) \geq n$. Indeed, if we take $\eta = \xi_1 \otimes ... \otimes \xi_n$ with $\xi_i \in E$, since $S_{\eta}^{(n)} = S_{\xi_1}^{(1)} \cdot ... \cdot S_{\xi_n}^{(1)}$ we get $\md(S_{\eta}^{(n)}) \geq \md(S_{\xi_1}^{(1)}) + ... + \md(S_{\xi_n}^{(1)}) \geq n$. Next, since the collection of elements $\eta':= \sum_{i=1}^{\ell} \xi^{(i)}_1 \otimes ... \otimes \xi^{(i)}_n$ is dense in $E^{\otimes n}$, and as we saw, $\Phi_m(\varphi(S_{\eta'}^{(n)})) = 0$ for all $m < n$, by continuity of $\varphi$ and $\Phi_m$ we get that $\Phi_m(\varphi(S_{\eta}^{(n)})) = 0$ for \emph{any} $\eta \in E^{\otimes n}$, so that $\md(\varphi(S_{\eta}^{(n)})) \geq n$ for any $\eta \in E^{\otimes n}$.

We now show that $\md(\varphi(T)) \geq \md(T)$ for any $T\in \tensor(E)$. Indeed, let $T\in \tensor(E)$ be an operator with $\md(T) = n \geq 0$. Then we can write $T = \sum_{k=n}^{\infty}\Phi_k(T)$ as a norm converging Cesaro sum, and by boundedness of $\varphi$, we obtain that $\varphi(T) = \sum_{k=n}^{\infty}\varphi(\Phi_k(T))$ converging Cesaro. Since $\Phi_k(T) = S_{\xi_k}^{(k)}$ is of minimal degree at least $k \geq n$, so would be $\varphi (\Phi_k(T))$. Then, by continuity of $\Phi_k$ and $\varphi$, we have that $\varphi(T)$ is of minimal degree at least $n$, and we see that $\md(\varphi(T)) \geq \md(T)$.

To show that $\md(\varphi(T)) = \md(T)$, we will show that $\md(\varphi^{-1}(S_{\xi}^{(1)})) \geq 1$ for $\xi \in F$ and bootstrap the above argument to show that $\md(\varphi^{-1}(T)) \geq \md(T)$, so that together with the above we get $\md(\varphi(T)) = \md(T)$. 

We show $\md(\varphi^{-1}(S_{\xi}^{(1)})) \geq 1$ for $\xi \in F$. Indeed, let $\xi \in F$ and write $\varphi^{-1}(S_{\xi}^{(1)}) = a + T$ with $a \in \Aa$ and $\md(T) \geq 1$. We have already shown that $\md(\varphi(T)) \geq 1$, so that
$$
\rho_{\varphi}(a) = \Phi_0(\varphi(a)) = \Phi_0(\varphi(a) + \varphi(T)) = \Phi_0(S_{\xi}^{(1)}) = 0
$$
Since $\varphi$ is base-detecting, $\rho_{\varphi}$ is a *-isomorphism, and we have that $a=0$. This means that $\md(\varphi^{-1}(S_{\xi}^{(1)})) \geq 1$, and we are done.
\end{proof}

We next prove an analogue of Proposition 6.17 of Section 6 in \cite{StochMat} in the discussion on semi-graded isomorphisms, that yields a reduction of our isomorphism problems.

\begin{proposition} \label{proposition:semi-graded-to-graded}
Let $E$ and $F$ be C*-correspondences over \emph{commutative} C*-algebras $\Aa$ and $\Bb$ respectively, and let $\varphi :\Tt_+(E) \rightarrow \Tt_+(F)$ be a semi-graded \emph{bounded} isomorphism. There is a unique bounded homomorphism $\widetilde{\varphi} : \Tt_+(E) \rightarrow \Tt_+(F)$ satisfying
$$
\widetilde{\varphi}(S^{(1)}_{\xi}) = \Phi_1 (\varphi(S^{(1)}_{\xi})), \ \ \xi \in E
$$
and $\widetilde{\varphi}$ is a graded completely bounded isomorphism such that $\widetilde{\varphi}^{-1} = \widetilde{\varphi^{-1}}$, and $\|\widetilde{\varphi} \|_{cb} \leq \|\varphi \| \cdot \| \varphi^{-1} \|$.
\end{proposition}

\begin{proof}
First note that since $\varphi$ is semi-graded, by Proposition \ref{prop:base-detect-semi-graded} it must be base-detecting. Hence, by Proposition \ref{proposition:criterion-semi-graded}, for any $T\in \Tt_+(E)$ with $\md(T)=n$ we must have $\Phi^F_n\varphi(T) = \Phi^F_n \varphi \Phi_n^E(T)$. It follows that for all $n\in \mathbb{N}$ and any $S\in \Tt_+(E)_n$ we must have
\begin{equation} \label{eq:inverse-relation}
S = \Phi^E_n(S) = \Phi^E_n\varphi^{-1}\varphi(S) = \Phi^E_n\varphi^{-1}\Phi^F_n\varphi (S)
\end{equation}
Set $\rho = \rho_{\varphi} = \Phi^F_0 \varphi |_{\Aa} : \Aa \rightarrow \Bb$, which is a *-isomorphism, and define a $\rho$-bimodule map $V_n :E^{\otimes n} \rightarrow F^{\otimes}$ by setting $V_n(\xi) = (\Psi^F_n)^{-1}\Phi^F_n\varphi \Psi^E_n(\xi)$, where $\Psi_n(\xi) = S^{(n)}_{\xi}$. Note that $V_n$ is clearly well-defined with $\|V_n \| \leq \|\varphi \|$ so that $V_n$ is a $\rho$-correspondence morphism. One similarly defines a $\rho^{-1}$-correspondence morphism $V_n' : F^{\otimes} \rightarrow E^{\otimes n}$ satisfying $\|V_n'\| \leq \|\varphi^{-1}\|$ that satisfies $V_n^{-1} = V_n'$ by equation \ref{eq:inverse-relation}. We now wish to show that $V=V_1$ is a $\rho$-similarity, so we show that $V$ is tensor power bounded by showing that $V_n = V^{\otimes n}$, and a similar argument would then work for $V^{-1}$. We show by induction that $V_n = V^{\otimes n}$. Indeed, suppose $V_k = V^{\otimes k}$ for all $k < n+m$ with $n,m \geq 1$. Let $\xi \in E^{\otimes n}$ and $\eta \in E^{\otimes m}$, then by semi-gradedness and the definition of $V_n$, $V_m$ and $V_{n+m}$, we have the following chain of equalities
$$
S^{(n+m)}_{V_{n+m}(\xi \otimes \eta)} = \Phi^F_{n+m} \varphi( S^{(n+m)}_{\xi \otimes \eta}) = (\Phi^F_n(\varphi(S^{(n)}_{\xi})) \Phi^F_m(\varphi(S^{(m)}_{\eta}))) =
$$
$$
S^{(n)}_{V_n(\xi)}S^{(m)}_{V_m(\eta)} = S^{(n)}_{V^{\otimes n}(\xi)}S^{(m)}_{V^{\otimes m}(\eta)} = S^{(n+m)}_{V^{\otimes (n+m)}(\xi \otimes \eta)}
$$
so that by applying $(\Psi^F_{n+m})^{-1}$ to both sides of this equation we obtain that $V_{n+m}(\xi \otimes \eta) = V^{\otimes (n+m)}(\xi \otimes \eta)$, so that $V_{n+m} = V^{\otimes (n+m)}$.

Thus, we have constructed a $\rho$-similarity $V: E \rightarrow F$ satisfying $S^{(1)}_{V(\xi)} = \Phi^F_1 \varphi (S^{(1)}_{\xi})$ for all $\xi \in E$ with the tensor iterates of $V$ and $V^{-1}$ bounded in norm by the norms of $\varphi$ and $\varphi^{-1}$ respectively. By item (1) of Theorem \ref{theorem:bounded-graded} the $\rho$-similarity $V$ promotes to a graded completely bounded isomorphism $\widetilde{\varphi} = Ad_V : \Tt_+(E) \rightarrow \Tt_+(F)$ uniquely determined by satisfying $S^{(1)}_{V(\xi)} = \widetilde{\varphi}(S^{(1)}_{\xi})$ for all $\xi \in E$, with $\|\widetilde{\varphi}\| \leq \|\varphi \| \| \varphi^{-1} \|$. So we see that $\widetilde{\varphi}(S^{(1)}_{\xi}) = \Phi^F_1\varphi(S^{(1)}_{\xi})$ for all $\xi \in E$ and that $\widetilde{\varphi}$ is uniquely determined by this property as required.
\end{proof}

\begin{corollary} \label{proposition:isometric-semi-graded-to-graded}
Let $E$ and $F$ be C*-correspondences over \emph{commutative} C*-algebras $\Aa$ and $\Bb$ respectively, and let $\varphi :\Tt_+(E) \rightarrow \Tt_+(F)$ be a semi-graded \emph{isometric} isomorphism. There is a unique bounded homomorphism $\widetilde{\varphi} : \Tt_+(E) \rightarrow \Tt_+(F)$ satisfying
$$
\widetilde{\varphi}(S^{(1)}_{\xi}) = \Phi_1 (\varphi(S^{(1)}_{\xi})), \ \ \xi \in E
$$
and $\widetilde{\varphi}$ is a graded completely isometric isomorphism such that $\widetilde{\varphi}^{-1} = \widetilde{\varphi^{-1}}$.
\end{corollary}

\section{Universal description and automatic continuity} \label{sec:univ-prop-auto-cont}

We start this section by proving a universal property for the Toeplitz algebra $\Tt(\sigma,w):= \Tt(C(\sigma,w))$ arising from a WPS $(\sigma,w)$, as a C*-algebra generated by certain set elements satisfying certain relations. This enables us to think of the non-self-adjoint tensor algebra $\tensor(\sigma,w) := \tensor(C(\sigma,w))$ as a the norm closed operator subalgebra of the universal C*-algebra $\Tt(\sigma,w)$ generated by the same set of elements. Moreover, we provide a criterion for automatic continuity, that will help answer the algebraic isomorphism problem for our tensor algebras, under the assumption that the union of $X_i$ covers $X$, where $X_i$ are the clopen domain of definition for $\sigma_i$.

\subsection{Universal description}
Recall that for a partial system $(\sigma,w)$, the positive operator $P(\sigma,w):C(X)\rightarrow C(X)$ used to construct the GNS C*-correspondence of $(\sigma,w)$ was given by
$$
P(\sigma,w)(f)(x) = \sum_{i : x\in X_i}w_i(x)f(\sigma_i(x))
$$

\begin{defi} \label{defi:representation-of-WPS}
Let $(\sigma,w)$ be a WPS on compact $X$. A \emph{representation} of $(\sigma,w)$ is a pair $(\pi, T)$ with $\pi: C(X) \rightarrow B(H)$ a unital *-representation and an operator $T\in B(H)$ such that
\begin{equation*} \label{eq:univ-prop}
T^*\pi(f)T = \pi \big(P(\sigma,w)(f) \big) \ \text{ for all } \ f\in C(X)
\end{equation*}

Denote by $C^*(\pi,T)$ the C*-algebra and $\overline{Alg}(\pi,T)$ the norm-closed algebra generated by the image of $\pi$ and $T$ inside $B(H)$. 
\end{defi}

The following universal description shows that we can think of $\tensor(\sigma,w)$ as a certain "semi-crossed product" by the positive map $P(\sigma,w)$.

\begin{theorem}[Universal description] \label{theorem:universal-description}
Let $(\sigma,w)$ be a WPS on compact $X$. Then the Toeplitz algebra $\Tt(\sigma,w)$ and the tensor algebra $\tensor(\sigma,w)$ are the universal C*-algebra and operator algebra (respectively) generated by a universal representation $(\pi_u, T_u)$ of $(\sigma,w)$.
\end{theorem}

\begin{proof}
Since representations of $(\sigma,w)$ are exactly representations of $(C(X), P(\sigma,w))$ in the sense of Definition 3.1 in \cite{ExelCroPro}, by Proposition 3.10 in \cite{ExelCroPro}, these are in bijection with \emph{isometric} (in the sense of Definition 2.11 in \cite{Muhly-Solel-Tensor-Rep}) representations $(\pi, \pi_{P(\sigma,w)})$ of the GNS C*-correspondence $GNS(\sigma,w)$, that satisfy $\pi_{P(\sigma,w)}(f \otimes g) = \pi(f) T \pi(g)$. By Theorem 2.12 in \cite{Muhly-Solel-Tensor-Rep}, these are in bijection with the representations $\tau_{(\pi, \pi_{P(\sigma,w)})}$ of $\Tt(\sigma,w)$ that send $\tensor(\sigma,w)$ to $\overline{Alg}(\pi(C(X)) \cup \pi_{P(\sigma,w)}(GNS(\sigma,w)))$. Hence, if $(\pi, T)$ is a representation of $(\sigma,w)$, it promotes to a representation $\tau_{(\pi, T)}$ of  $\Tt(\sigma,w)$ that sends $\tensor(\sigma,w)$ to $\overline{Alg}(\pi, T)$, and every such representations $\pi$ of $\Tt(\sigma,w)$ arises in this way, and must send $\tensor(\sigma,w)$ to $\overline{Alg}(\pi, T)$.
\end{proof}

\subsection{Automatic continuity}
We now wish to show that under certain conditionson a WPS, an algebraic homomorphism onto $\tensor(\sigma,w)$ is automatically bounded. We will follow the ideas of Davidson, Donsig, Hudson, Katsoulis and Kribs used in \cite{IsoConjAlg, Donsig-Hudson-Katsoulis, IsoDirGrAlg}.
For Banach algebras $\Aa$ and $\Bb$ suppose we have a surjective homomorphism $\varphi: \Aa \rightarrow \Bb$. Let
$$
\Ss(\varphi) = \{ \ b\in \Bb \ | \ \text{there is a sequence } (a_n) \text{ in } \Aa \text{ with } a_n \rightarrow 0 \text{ and } \varphi(a_n) \rightarrow b \ \}
$$
It is readily verified that the graph of $\varphi$ is closed if and only if $\Ss(\varphi) = \{ 0\}$, hence, by the closed graph theorem $\varphi$ is continuous if and only if $\Ss(\varphi) = \{0\}$. The following first appeared in \cite{Donsig-Hudson-Katsoulis} as an adaptation of a lemma by Sinclair, the origins of which can be traced back to \cite{Sinclair}.

\begin{lemma}[Sinclair]
Let $\Aa$ and $\Bb$ be Banach algebras and $\varphi : \Aa \rightarrow \Bb$ be a surjective algebraic homomorphism. Let $(b_n)_{n\in \nn}$ be any sequence in $\Bb$. Then there exists $N\in \nn$ such that for all $n\geq N$,
$$
\overline{b_1 b_2 ... b_n \Ss(\varphi)} = \overline{b_1 b_2 ... b_N \Ss(\varphi)} \ \text{ and } \ \overline{\Ss(\varphi) b_n ... b_2 b_1} = \overline{\Ss(\varphi) b_N ... b_2 b_1}
$$
\end{lemma} 

For every WPS $(\sigma,w)$ we can define the weight function of the system to be $w_{\sigma}(x) = P(\sigma,w)(1)(x) = \sum_{i : x \in X_i} w_i(x)$ which is positive continuous function that vanishes only on $X - \cup_{i=1}^d X_i$. 

\begin{defi}
Let $\sigma$ be a partial system on $X$. We say $\sigma$ is well-supported if $\{X_i\}$ covers $X$, where $X_i$ are the clopen domain of definition for $\sigma_i$.
\end{defi}

When we have a well-supported $(\sigma,w)$, we define the normalized WPS $(\sigma,\widetilde{w})$ by setting
$\widetilde{w} = (\frac{w_1}{w_{\sigma}},...,\frac{w_n}{w_{\sigma}})$, and we say that $(\sigma,w)$ is normalized if $w_{\sigma} = 1$. Note that when $(\sigma,w)$ is a well-supported normalized system, we must have that $P(\sigma,w)$ is a unital map, or in other words a Markov-Feller operator, and so for every representation $(\pi, T)$ of $(\sigma,w)$ with unital $\pi$,  we have that $T^*T = T^*\pi(1) T = \pi(P(\sigma,w)(1)) = 1$ and hence $T$ must be an isometry.

\begin{theorem} \label{theorem:automatic-continuity}
Let $(\sigma,w)$ and $(\tau,u)$ be WPS operating on $X$ and $Y$ respectively such that either $\sigma$ or $\tau$ are well-supported. Then every algebraic isomorphism $\varphi :\tensor(\sigma,w) \rightarrow \tensor(\tau,u)$ is automatically a bounded isomorphism.
\end{theorem}

\begin{proof}
Suppose without loss of generality that $\tau$ is well-supported. Since for every edge $e\in Gr(\tau)$ we have that $\frac{\widetilde{u}}{u}(e) = u_{\tau}(s(e))^{-1}$, we see that $\frac{\widetilde{u}}{u}$ is continuous on $Gr(\tau)$ so that $(\tau,u)$ and $(\tau,\widetilde{u})$ are branch-transition conjugate. By Corollary \ref{cor:branch-transition-char} and Theorem \ref{theorem:isometric-graded} used in tandem, $\tensor(\tau,u)$ is graded completely isometrically isomorphic to $\tensor(\tau,\widetilde{u})$. So we assume without loss of generality that $(\tau,u)$ is normalized. In this case, the constant function $1 = 1 \odot 1 \in C(Gr(\tau))$ gives rise to an isometry $W_{(\tau,u)}:= S^{(1)}_{1} = S^{(1)}_{1 \odot 1} \in \Ker \Phi_0 \subset \tensor(\tau,u)$, since $(\tau,u)$ is well-supported and normalized.

Now suppose towards contradiction that there is $0 \neq T \in \Ss(\varphi)$. Since $W_{(\tau,u)}$ is an isometry, we have that $W_{(\tau,u)}^n T \neq 0$ for all $n\in \nn$.
By Sinclair's lemma there is some $N \in \nn$ such that for all $n \geq N$ we have
$$
W_{(\tau,u)}^N\Ss(\varphi) = W_{(\tau,u)}^n \Ss(\varphi) \subset \bigcap_{k<n}\Ker \Phi_k
$$
So in fact we must have that $W_{(\tau,u)}^N\Ss(\varphi) = \cap_{k\in \nn} \Ker \Phi_k = \{0 \}$, in contradiction to having $W_{(\tau,u)}^NT \neq 0$ as shown above.
\end{proof}

\section{Character space} \label{sec:character-space}

In this subsection, we adapt the methods of Hadwin and Hoover \cite{ConjTrans}, which were used in the solution of the conjugacy problem \cite{IsoConjAlg}, to compute the character space of $\tensor(\sigma,w)$ for any WPS $(\sigma,w)$. We then use this to show that every algebraic isomorphism $\varphi : \tensor(\sigma,w) \rightarrow \tensor(\tau,u)$ is automatically base-detecting, so that the base spaces $X$ and $Y$ can be identified. Finally, we provide a criterion to detect semi-gradedness from the induced homeomorphism on the character spaces.

\subsection{Computing the character space}
Let $(\sigma,w)$ be a WPS on compact $X$. Denote by $\Mm(\sigma,w)$ the space of multiplicative linear functionals on $\Tt_+(\sigma,w)$ with its weak* topology. $\Mm(\sigma,w)$ is partitioned by $X$ since for every $\theta \in \Mm(\sigma,w)$ there is a unique $x\in X$ such that $\theta | _{C(X)} = \delta_x$. We denote by $\Mm(\sigma,w)_x$ the weak* closed subset of $\theta \in \Mm(\sigma,w)$ satisfying $\theta | _{C(X)} = \delta_x$, and we let $\theta_{x,0}$ be the unique element in $\Mm(\sigma,w)_x$ such that $\theta_{x,0}(\Ker \Phi^P_0) = \{0\}$. Denote by $W_{(\sigma,w)} := S^{(1)}_{1} = S^{(1)}_{1\odot 1}$, the shift operator by the constant function $1 = 1 \odot 1 \in C(\sigma,w)$. Note that since $\Ker \Phi^P_0$ is the closed two sided ideal generated by $W_{(\sigma,w)}$, we have that $\theta_{x,0}$ is the unique element in $\Mm(\sigma,w)_x$ such that $\theta_{x,0}(W_{(\sigma,w)}) = 0$. 

\begin{defi}
Let $\sigma$ be a $d$-variable partial system on compact $X$. We say that $x\in X$ is a \emph{fixed point} for $\sigma$ if $\sigma_i(x) = x \in X_i$ for some $1\leq i \leq d$. We denote by $\Fix(\sigma)$ the closed set of fixed points of $\sigma$.
\end{defi}

\begin{lemma} \label{lemma:fixed-point-char}
Let $(\sigma,w)$ be a WPS on compact $X$, $x\in X$, $\theta \in \Mm(\sigma,w)_x$ and $\xi \in C(\sigma,w)$. Then we have
\begin{enumerate}
\item
If $x\in \Fix(\sigma)$ then $ \theta(S^{(1)}_{\xi}) = \xi(x,x) \theta(W_{(\sigma,w)})$
\item
If $x\notin \Fix(\sigma)$ then $\theta(S^{(1)}_{\xi}) = 0$
\end{enumerate}
In particular, when $x \notin \Fix(\sigma)$, we have $\Mm(\sigma,w)_x = \{ \theta_{x,0} \}$.
\end{lemma}

\begin{proof}
First we show (1). Let $x\in X$ be a fixed point for $\sigma$. For every open neighborhood $U$ of $x$, by Urysohn's Lemma, there is a continuous function $f_U:X \rightarrow [0,1]$ with $f_U(x) = 1$ and $f_U(y) = 0$ for $y \notin U$.
Thus, for $\theta \in \Mm(\sigma,w)_x$ and $U,V$ open neighborhoods of $x$ we have,
$$
|\theta(S^{(1)}_{\xi} - \xi(x,x) W_{(\sigma,w)})|^2 = |\theta(f_{U} \cdot (S^{(1)}_{\xi} - \xi(x,x) W_{(\sigma,w)})) \cdot f_{V})|^2 = |\theta(S_{f_{U} \cdot (\xi  - \xi(x,x) 1)  \cdot f_{V}})|^2 \leq
$$
$$ 
\| f_{U}  \cdot (\xi  - \xi(x,x)1)  \cdot f_{V} \|^2 = \sup_{y \in X} \sum_{i: y \in X_i} |f_{U}(\sigma_i(y))|^2 |\xi(\sigma_i(y),y)  - \xi(x,x)|^2 |f_{V}(y)|^2 w_i(y) \leq
$$
$$
\sup_{y \in V} \sum_{i : y \in X_i} |f_{U}(\sigma_i(y))|^2 |\xi(\sigma_i(y),y)  - \xi(x,x)|^2w_i(y)
$$
Taking infimum over all open neighborhoods $V$ of $x$ we get
$$
|\theta(S^{(1)}_{\xi} - \xi(x,x) W_{(\sigma,w)})|^2 \leq \sum_{i : x\in X_i} |f_{U}(\sigma_i(x))|^2 |\xi(\sigma_i(x),x)  - \xi(x,x)|^2 w_i(x) \leq 
$$
$$
\sum_{i: \ x\in X_i, \ \sigma_i(x)\in U} |\xi(\sigma_i(x),x)  - \xi(x,x)|^2 w_i(x)
$$
Taking infimum over all $U$ open neighborhoods of $x$, we obtain 
$$
|\theta(S^{(1)}_{\xi} - \xi(x,x) W_{(\sigma,w)})|^2 \leq \sum_{i: \sigma_i(x) = x \in X_i} |\xi(\sigma_i(x),x)  - \xi(x,x)|^2 w_i(x) = 0
$$
and we must have that $\theta(S^{(1)}_{\xi}) = \xi(x,x) W_{(\sigma,w)}$.

In order to show (2), note that if $x\notin \Fix(\sigma)$, a similar chain of inequalities, replacing $\xi(x,x)W_{(\sigma,w)}$ by $0$ above, would yield that for all $\theta \in \Mm(\sigma,w)_x$ we have $\theta(S^{(1)}_{\xi}) = 0$.

Finally, if $x\notin \Fix(\sigma)$, we have that $\theta(W_{(\sigma,w)}) = 0$ for all $\theta \in \Mm(\sigma,w)_x$ so that $\theta(\Ker \Phi_0) = 0$ for all $\theta \in \Mm(\sigma,w)_x$. Now since $\theta_{x,0}$ is the only element in $\Mm(\sigma,w)_x$ with $\theta_{x,0}(\Ker \Phi_0) = 0$, we must then have that $\theta = \theta_{x,0}$ and $\Mm(\sigma,w)_x = \{ \theta_{x,0}\}$.
\end{proof}

Now, in the case where $x\in X$ is a fixed point for $\sigma$, we are interested to know how $\theta \in \Mm(\sigma,w)_x$ acts on iterates $S^{(n)}_{\xi}$ for $\xi\in C(\sigma,w)^{\otimes n} \cong C(Gr(\sigma^n))$. Recall the discussion preceding Proposition \ref{proposition:iterate-computation} where we identified $Gr(\sigma^n)$ with the collection of orbits of length $n+1$ inside $X^{n+1}$, that is the collection of sequences $(x_{n+1},..., x_1)$ such that for every $1\leq m \leq n$ there is some $1\leq i \leq d$ with $\sigma_i(x_m) = x_{m+1} \in X_i$.

Thus, take $\xi^{(1)},..., \xi^{(n)} \in C(\sigma,w)$ and note that by Lemma \ref{lemma:fixed-point-char},
$$\theta(S^{(n)}_{\xi^{(1)} \otimes ...\otimes \xi^{(n)}}) = \theta(S^{(1)}_{\xi^{(1)}}) \cdot ... \cdot \theta(S^{(1)}_{\xi^{(n)}}) = \xi^{(n)}(x,x) \cdot ... \cdot \xi^{(1)}(x,x) \cdot \theta(W_{(\sigma,w)})^n$$ 
By supremum norm approximation we obtain for every $\xi \in C(Gr(\sigma^n))$ that
$$
\theta(S^{(n)}_{\xi}) = \xi(x,...,x) \cdot \theta(W_{(\sigma,w)})^n
$$
due to density of the linear span of elements of the form $\xi^{(1)} \otimes ...\otimes \xi^{(n)}$ in $C(\sigma,w)^{\otimes n} \cong C(Gr(\sigma^n))$, with the supremum norm, established by Proposition \ref{proposition:iterate-computation}.

The next proposition is an adaptation of the methods of Section 3 in \cite{IsoConjAlg}, originally used by Hadwin and Hoover in \cite{ConjTrans}. For a WPS $(\sigma,w)$, recall that we defined the weight of an edge $(y,x) \in Gr(\sigma)$ to be $w(y,x) = \sum_{i : \sigma_i(x) = y, \ x\in X_i}w_i(x)$.

\begin{proposition} \label{proposition:analytic-discs}
Let $X$ be a compact space, $(\sigma,w)$ a WPS on $X$ and $x\in \Fix(\sigma)$. Then $\Mm(\sigma,w)_x \cong \overline{\dd}_{r^w_x}$ via the map $\theta \mapsto \theta(W_{(\sigma,w)})$, where $\overline{\dd}_{r^w_x}$ is the closed disc of radius $r^w_x =\sup_{\theta \in \Mm(\sigma,w)_x} |\theta(W_{(\sigma,w)})| = \sqrt{w(x,x)}$.

Moreover if $\Theta_x : \overline{\dd}_{r^w_x} \rightarrow \Mm(\sigma,w)_x$ is the homeomorphism above, then it is in fact pointwise analytic on $\dd_{r^w_x}$, in the sense that for every $T\in \Tt_+(\sigma,w)$, the function $\Theta_x(\cdot)(T): \dd_{r^w_x} \rightarrow \cc$ is analytic.
\end{proposition}

\begin{proof}
We first define a character $\theta_{x,z}$ for every $z\in \cc$ with $|z| < \sqrt{w(x,x)}$. Let $T\in \Tt_+(\sigma,w)$. By Proposition \ref{proposition:fourier-series} we get that $T$ has a Fourier series representation as $T = \sum_{n = 0}^{\infty}S^{(n)}_{\xi_n}$ converging Cesaro. We then define 
$$
\theta_{x,z}(T) = \sum_{n=0}^{\infty}\xi_n(x,...,x)z^n 
$$ 
Since $|z| < \sqrt{w(x,x)}$ and $|\xi_n(x,...,x)| \leq \frac{\|\xi_n \|}{\sqrt{w(x,x)^n}} = \frac{\| \Phi_n(T) \|}{\sqrt{w(x,x)^n}}$ we get
$$
|\theta_{x,z}(T)| \leq \sum_{n=0}^{\infty}|\xi_n(x,...,x)| |z|^n \leq \sum_{n=0}^{\infty} \|\Phi_n(T)\| \Big(\frac{|z|}{\sqrt{w(x,x)}}\Big)^n \leq
\| T \| \sum_{n=0}^{\infty} \Big(\frac{|z|}{\sqrt{w(x,x)}}\Big)^n 
$$
so that the above is a well-defined multiplicative linear functional on $\Tt_+(\sigma,w)$. Indeed, $\theta_{x,z}$ is linear and multiplicative due to multiplication of Fourier series given in Proposition \ref{proposition:fourier-series} and due to the identification of Proposition \ref{proposition:iterate-computation}.

We show that for every $\theta \in \Mm(\sigma,w)_x$, one must have $|\theta(W_{(\sigma,w)})| \leq \sqrt{w(x,x)}$. Indeed, for every open neighborhood of $x$, by Urysohn's Lemma, there is a continuous function $f_U : X \rightarrow [0,1]$ with $f_U(x) = 1$ and $f(y) = 0$ for $y\notin U$. Thus, for $U,V$ open neighborhoods of $x$ we have,
$$
|\theta(W_{(\sigma,w)})|^2 = |\theta(f_U \cdot W_{(\sigma,w)} \cdot f_V)|^2 = |\theta(S_{f_U \odot f_V})|^2 \leq \| S_{f_U \odot f_V} \|^2 =
$$
$$
\sup_{x\in X} \sum_{i : x\in X_i} |f_U(\sigma_i(x))|^2 |f_V(x)|^2 w_i(x) \leq \sup_{x\in V} \sum_{i : x\in X_i} |f_U(\sigma_i(x))|^2  w_i(x)
$$
Taking infimum over all open neighborhoods $V$ of $x$ we get that
$$
|\theta(W_{(\sigma,w)})|^2 \leq \sum_{i : x\in X_i} |f_U(\sigma_i(x))|^2 w_i(x) \leq 
\sum_{i : x\in X_i, \ \sigma_i(x) \in U} w_i(x)
$$
Taking infimum over all $U$ open neighborhoods of $x$, we obtain
$$
|\theta(W_{(\sigma,w)})|^2 \leq \sum_{i : \sigma_i(x) = x \in X_i} w_i(x) = w(x,x)
$$

Thus, we see that $|\theta(W_{(\sigma,w)})| \leq \sqrt{w(x,x)}$ and so the range of the map $\theta \mapsto \theta(W_{(\sigma,w)})$ contains the open disc $\dd_{r^w_x}$ which is dense in $\overline{\dd}_{r^w_x}$. 

Hence, the function from $\Mm(\sigma,w)_x$ to $\overline{\dd}_r$ given by 
$\theta \mapsto \theta(W_{(\sigma,w)})$ is a continuous injective map between compact spaces that has dense range, and thus must be a homeomorphism. 

For the last part, we see that the inverse of the above homeomorphism restricted to the open disc $\Theta_x : \dd_{r_x^w} \rightarrow \Mm(\sigma,w)_x$ is given by 
$$
\Theta_x(z)(T) = \theta_{x,z}(T) = \sum_{n=0}^{\infty}\xi_n(x,...,x)z^n
$$
for $T \in \Tt_+(\sigma,w)$ with Fourier series $T = \sum_{n = 0}^{\infty}S^{(n)}_{\xi_n}$. So that $\Theta_x(\cdot)(T)$ is analytic on $\dd_{r^w_x}$ for every fixed $T\in \Tt_+(\sigma,w)$.

\end{proof}

Let us call a subset of $\Mm(\sigma,w)$ an \emph{analytic disc} if it is the range of a pointwise analytic injective map $\Theta : \dd_s \rightarrow \Mm(\sigma,w)$, for $s>0$. For $f\in C(X)$ we must have that $\Theta(z)(\overline{f}) = \overline{\Theta(z)(f)}$, since for every $z\in \dd_s$ there is some $x\in X$ such that $\Theta(z) | _{C(X)} = \delta_x$. Thus, due to analyticity, $\Theta(\cdot)(f): \dd_s \rightarrow \cc$ must be constant $f(x)$, and so $\Theta(\dd_s)$ is contained in $\Mm(\sigma,w)_x$ for some $x \in X$. Proposition \ref{proposition:analytic-discs} tells us that for every fixed point $x\in X$ of $\sigma$, the interior of $\Mm(\sigma,w)_x$ is an analytic disc, and is hence maximal in the collection of analytic discs, due to the above observation and the fact that every analytic disc contained in $\Mm(\sigma,w)_x$ must be open due to the Open Mapping Theorem, and is hence contained in $\Theta_x(\dd_{r^w_x})$.

\subsection{Base-detection}
It turns out we can use maximal analytic discs together with the computation of the character space to obtain automatic base-detection for isomorphisms between tensor algebras associated to WPS. For any linear homomorphism $\theta$ between Banach algebras $\Aa$ and $\Bb$, we denote by $\theta^* : \Mm_{\Bb} \rightarrow \Mm_{\Aa}$ the map induced between their character spaces.

\begin{proposition} \label{proposition:base-detection}
Let $(\sigma,w)$ and $(\tau,u)$ be WPS on compact $X$ and $Y$ respectively and let $\varphi : \Tt_+(\sigma,w) \rightarrow \Tt_+(\tau,u)$ be an algebraic isomorphism. Then $\varphi$ is base-detecting and in fact, $\rho_{\varphi}^*$ is a bijection that sends fixed points of $\tau$ to those of $\sigma$.
\end{proposition}

\begin{proof}
Let $\varphi$ be as in the statement of the proposition. Then $\varphi$ induces a homeomorphism $\varphi^* : \Mm(\tau,u) \rightarrow \Mm(\sigma,w)$. It is easily verified that $\varphi^*$ sends maximal analytic discs to maximal analytic discs, since it preserves the lattice of inclusion of analytic discs. Hence we obtain a bijection between the maximal analytic discs of $\Mm(\tau,u)$ and $\Mm(\sigma,w)$ which extends to a bijection between closures of such analytic discs. That is, to every $y\in Y$ there is a unique $\gamma_{\varphi}(y) \in X$ such that $\varphi^*$ restricted to $\Mm(\tau,u)_y$ is a homeomorphism onto $\Mm(\sigma,w)_{\gamma_{\varphi}(y)}$, and furthermore, we must have that $\gamma_{\varphi^{-1}} = \gamma_{\varphi}^{-1}$ and that $\gamma_{\varphi}$ is a bijection between fixed points of $\tau$ and fixed points of $\sigma$. 

To show that $\varphi$ is base detecting, let $\iota_X : C(X) \rightarrow \Tt_+(\sigma,w)$ be the canonical inclusion. By noting that $\iota^* : \Mm(\sigma,w) \rightarrow X$ is the canonical quotient map sending every element in $\Mm(\sigma,w)_x$ to $\theta_{x,0}$ (which is identified with $x\in X$), that $\Phi_0^* : Y \rightarrow \Mm(\tau,u)$ is the map $\Phi_0^*(y) = \theta_{y,0}$ and that
$$
\gamma_{\varphi} = \iota^* \circ \varphi^* \circ \Phi_0^* = (\Phi_0 \circ \varphi \circ \iota)^* = \rho_{\varphi}^*
$$
we see that $\rho_{\varphi}$ is a *-isomorphism satisfying  $\rho^{-1}_{\varphi} = \rho_{\varphi^{-1}}$ by using the commutative Gelfand-Naimark functorial duality, with $\rho_{\varphi}^* = \gamma_{\varphi}$ inducing a bijection between the fixed points of $\tau$ and those of $\sigma$.
\end{proof}

Proposition \ref{proposition:base-detection} enables an important reduction of our isomorphisms problems. Indeed, if $(\sigma,w)$ and $(\tau,u)$ are WPS on $X$ and $Y$ respectively and $\varphi : \tensor(\sigma,w) \rightarrow \tensor(\tau,u)$ is a  bounded / isometric isomorphism. Let $\gamma = (\rho^{-1}_{\varphi})^* : X \rightarrow Y$ be the induced map on the base spaces. Then obviously the WPS $(\tau^{\gamma},u^{\gamma})$ is conjugate to $(\tau,u)$ via $\gamma^{-1}$, and so one can see that $(\sigma,w)$ is weighted-orbit / branch-transition conjugate to $(\tau,u)$ via $\gamma$ if and only if $(\sigma,w)$ is weighted-orbit / branch-transition conjugate to $(\tau^{\gamma},u^{\gamma})$ via $id_X$ respectively. Moreover, the conjugation between $(\tau^{\gamma},u^{\gamma})$ and $(\tau,u)$ promotes to a completely isometric graded isomorphism $\widetilde{\gamma} : \tensor(\tau^{\gamma},u^{\gamma}) \rightarrow \tensor (\tau,u)$ and so $\psi = \widetilde{\gamma}^{-1} \circ \varphi : \tensor(\sigma,w) \rightarrow \tensor(\tau^{\gamma},u^{\gamma})$ is a bounded / isometric isomorphism (resp. to what $\varphi$ is) where the WPS $(\sigma,w)$ and $(\tau^{\gamma},u^{\gamma})$ are on the \emph{same space} $X$ with $\rho^*_{\psi} = Id_X$. Our goal is then reduced to establishing weighted-orbit / branch-transition conjugation of $(\sigma,w)$ and $(\tau^{\gamma},u^{\gamma})$ via $Id_X$ from a bounded / isometric isomorphism $\psi : \tensor(\sigma,w) \rightarrow \tensor(\tau^{\gamma},u^{\gamma})$ respectively, with $\rho^*_{\psi} = Id_X$. This motivates the following definition

\begin{defi}
Let $(\sigma,w)$ and $(\tau,u)$ be partial systems on $X$. We say that an isomorphism $\varphi : \tensor(\sigma,w) \rightarrow \tensor(\tau,u)$ \emph{covers} $X$ if $\rho^*_{\varphi} = Id_X$ which is equivalent to having $\rho_{\varphi} = \Phi_0 \circ \varphi |_{C(X)} = Id_{C(X)}$.
\end{defi}

\subsection{Semi-gradedness}

Next, we characterize semi-graded isomorphisms between tensor algebras arising from WPS, in terms of the induced homeomorphism on the character spaces, and show how this can be used to produce a semi-graded isomorphism from a general one, for WPS comprised of strict contractions on compact perfect metric spaces.

If $(\sigma,w)$ and $(\tau,u)$ are WPS on $X$, and $\varphi: \Tt_+(\sigma,w) \rightarrow \Tt_+(\tau,u)$ is an algebraic isomorphism covering $X$, by Proposition \ref{proposition:base-detection} we must have that $f_x^{\varphi}(\cdot) := \varphi^*(\theta_{x, \cdot})(W_{(\sigma,w)}) : \Mm(\tau,u)_x \rightarrow \Mm(\sigma,w)_x$. If $x \notin \Fix(\tau) = \Fix(\sigma)$ we must have that $f_x^{\varphi}(\theta_{x,0}) = \theta_{x,0}$ since $\Mm(\sigma,w)_x = \{ \theta_{x,0} \}$.

Next, we note that if $x \in \Fix(\tau)$ is not an interior point of $\Fix(\tau)$ in $X$ then $f_x^{\varphi}(0) = 0$ due to continuity of $\varphi^*(\theta_{x,0})(W_{(\sigma,w)})$ in $x \in X$, and the fact that for points $x' \notin \Fix(\tau)$ we have $\varphi^*(\theta_{x',0}) = \theta_{x',0}$. Hence, the only "problematic" points are those in the interior of $\Fix(\tau)$. Thus we obtain the following characterization of semi-gradedness.

\begin{proposition} \label{proposition:semi-graded-character-space}
Let $(\sigma,w)$ and $(\tau,u)$ be WPS on compact $X$. A bounded isomorphism $\varphi : \Tt_+(\sigma,w) \rightarrow \Tt_+(\tau,u)$ covering $X$ is semi-graded if and only if $f_x^{\varphi}(0) = 0$ for all $x$ in the interior of $\Fix(\tau)$. In particular, if either $\Fix(\sigma)$ or $\Fix(\tau)$ have empty interior, then every isomorphism $\varphi$ is semi-graded.
\end{proposition}

\begin{proof}
If $\varphi$ is semi-graded, then $\varphi(W_{(\sigma,w)}) \in \Ker \Phi_0$ and so 
$$
f_x^{\varphi}(0) = \varphi^*(\theta_{x,0})(W_{(\sigma,w)}) = \theta_{x,0} (\varphi(W_{(\sigma,w)})) = \theta_{x,0} (\Phi_0(\varphi(W_{(\sigma,w)}))) = 0
$$
and so $f_x^{\varphi}(0) = 0$.

Conversely, if $f_x^{\varphi}(0) = 0$ for all $x \in X$ and $\varphi$ covers $X$, we have that $\varphi^*(\theta_{x,0}) = \theta_{x,0}$ for all $x\in X$, and by Proposition \ref{proposition:criterion-semi-graded} it suffices to show that for any $\xi \in C(\sigma,w)$ we have $\md(\varphi(S^{(1)}_{\xi})) \geq 1$. Indeed, write $\varphi(S^{(1)}_{\xi}) = h + T$ with $\md(T) \geq 1$ and $h\in C(X)$. Since for $x\in X$ we have that $h(x) = \theta_{x,0} (\varphi(S^{(1)}_{\xi})) = \theta_{x,0}(S^{(1)}_{\xi}) = 0$, we are done.
\end{proof}

As a corollary to the above, we show that every isomorphism is automatically semi-graded between tensor algebras arising from distributed iterated function systems and graph-directed systems as in Subsections \ref{subsec:dfs} and \ref{subsec:gds} respectively, when the spaces are with no isolated points.

\begin{corollary}
Let $(\sigma,w)$ and $(\tau,u)$ be $d$-variable and $d'$-variable WPS on a metric compact \emph{perfect} spaces $X$ and $Y$ respectively, such that either $\sigma$ or $\tau$ is comprised of strict contractions. If $\varphi : \tensor(\sigma,w) \rightarrow \tensor(\tau,u)$ is a bounded / isometric isomorphism, then it is automatically semi-graded.
\end{corollary}

\begin{proof}
Without loss of generality, $\sigma$ is comprised of contractions. Let $r$ be the metric on $X$. Since for any two points $x,y\in X$ we must have $r(\sigma_i(x),\sigma_i(y)) < r(x,y)$ for all $i\in \{1,...,d\}$, we see that $\sigma_i$ can have at most one fixed point, and so $\Fix(\sigma)$ has at most $d$ points. Since $X$ is perfect, $\Fix(\sigma)$ must have empty interior, and so by Proposition \ref{proposition:semi-graded-character-space} $\varphi$ must be semi-graded.
\end{proof}

\section{Isomorphisms of tensor algebras arising from WPS} \label{sec:main-results}

In this section we adapt a new method in the analysis of character spaces due to Davidson, Ramsey and Shalit in \cite{DRS}, and use this to construct a bounded / isometric semi-graded isomorphism from any bounded / isometric isomorphism of our tensor algebras respectively. We then use this to provide two theorems that separately deal with classification up to bounded isomorphism and classification up to isometric isomorphism, which turn out to yield two distinct equivalences.

\subsection{Reduction to the semi-graded case}
Let $(\sigma,w)$ be a WPS on $X$. Recall the gauge group action $\alpha : \mathbb{T} \rightarrow Aut(\tensor(\sigma,w))$ uniquely determined on generators by $\alpha_{\lambda}(S_{\xi}^{(1)}) = \lambda S_{\xi}^{(1)}$ and $\alpha_{\lambda}(f) = f$ for $\xi \in C(\sigma,w)$ and $f\in C(X)$. Now, if $(\sigma,w)$ and $(\tau,u)$ are WPS on $X$, and $\varphi: \Tt_+(\sigma,w) \rightarrow \Tt_+(\tau,u)$ is an algebraic isomorphism covering $X$, by Proposition \ref{proposition:base-detection} we have that $\varphi^* |_{\Mm(\tau,u)_x} : \Mm(\tau,u)_x \rightarrow \Mm(\sigma,w)_x$. 

Next, if $x \in \Fix(\tau)$, by Proposition \ref{proposition:analytic-discs} we can identify $\varphi^* |_{\Mm(\tau,u)_x}$ with a bijective biholomorphism $f_x^{\varphi} := \Theta_x^{-1} \circ \varphi^* \circ \Theta_x: \dd_{r_x^u} \rightarrow \dd_{r_x^w}$ which then must be of the form given by $f_x^{\varphi}(z) = r_x^w \hat{f}_x^{\varphi}( (r_x^u)^{-1} z)$ where $\hat{f}_x^{\varphi}$ is a biholomorphism of the \emph{unit} disc $\dd$ given by
$$
\hat{f}_x^{\varphi}(z) = e^{i\theta_x}\frac{w_x - z}{1 - \overline{w_x}z}
$$
for some $\theta_x \in [0,2\pi]$ and $w_x \in \dd$. Note also that since $f_x^{\varphi}(0) = \varphi^*(\theta_{x, 0})(W_{(\sigma,w)}) = r_x^w e^{i\theta_x}w_x$, and since $\varphi^*(\theta_{x, 0})(W_{(\sigma,w)})$ depends continuously on $x\in X$, we can extend $f_x^{\varphi}$ continuously to be $0$ for $x \notin \Fix(\tau)$. Further, if $\psi : \tensor(\tau,u) \rightarrow \tensor(\pi,v)$ is another algebraic isomorphism covering $X$ we have that $\hat{f}_x^{\psi \circ \varphi} = \hat{f}_x^{\varphi} \circ \hat{f}_x^{\psi}$.

We now wish to examine an isomorphism $\varphi : \tensor(\sigma,w) \rightarrow \tensor(\tau,u)$ covering $X$ for which there exists $x\in X$ an interior point of $\Fix(\sigma) = \Fix(\tau)$, with $f_x^{\varphi}(0) \neq 0$.

Fix an element $x\in \Fix(\tau)$ with $f_x^{\varphi}(0) \neq 0$. One can then find $\lambda_x, \gamma_x \in \mathbb{T}$ such that the isomorphism $\psi = \varphi \circ \alpha_{\lambda_x} \circ \varphi^{-1} \circ \alpha_{\gamma_x}  \circ \varphi$ satisfies $f_x^{\psi}(0) = 0$. Indeed, for $\lambda \in \mathbb{T}$, since
$ \hat{f}_x^{\varphi \circ \alpha_{\lambda}}(0) = \lambda \cdot \hat{f}_x^{\varphi}(0)$, we get that
$C = \{ \hat{f}_x^{\varphi \circ \alpha_{\lambda}}(0) | \lambda \in \mathbb{T} \}$
is a circle centered around $0$. Since $\hat{f}_x^{\varphi^{-1}}$ is a Mobius map of the form described above, it must send $C$ to a circle \emph{through} the origin. That is,
$C' = \hat{f}_x^{\varphi^{-1}}(C) = \{\hat{f}_x^{\varphi \circ \alpha_{\lambda} \circ \varphi^{-1}}(0) | \lambda \in \mathbb{T} \}$ is a circle through the origin, since for $\lambda = 0$ we get $0 = \hat{f}_x^{Id}(0) = \hat{f}_x^{\varphi \circ \alpha_{\lambda} \circ \varphi^{-1}}(0) \in C'$. If again take arbitrary $\gamma \in \mathbb{T}$ and do this, we can "fill the circle". That is, since $\hat{f}_x^{\varphi^{-1}\circ \alpha_{\gamma}}(C) = \gamma \cdot f_x^{\varphi^{-1}}(C) = \gamma \cdot C'$, the region bounded by $C'$, which we denote by $ins(C')$, is a subset of $\{ \ \hat{f}_x^{\varphi \circ \alpha_{\lambda} \circ \varphi^{-1} \circ \alpha_{\gamma}}(0) \ | \ \lambda, \gamma \in \mathbb{T} \ \}$. Once more, since $\hat{f}_x^{\varphi}$ is the inverse of $\hat{f}_x^{\varphi^{-1}}$, being a Mobius map, it must send $C'$ back to $C$, and so it must send $ins(C')$ to $ins(C)$. Thus we obtain that the set
$$
\{ \ \hat{f}_x^{\varphi \circ \alpha_{\lambda} \circ \varphi^{-1} \circ \alpha_{\gamma}  \circ \varphi}(0) \ | \ \lambda, \gamma \in \mathbb{T} \ \}
$$
contains the origin, and hence there is some choice of $\lambda_x$ and $\gamma_x$ with which 
$$
\hat{f}_x^{\varphi \circ \alpha_{\lambda_x} \circ \varphi^{-1} \circ \alpha_{\gamma_x}  \circ \varphi}(0) = 0
$$
We now wish to show that a choice of \emph{continuous} functions $x \mapsto \lambda_x$ and $x \mapsto \gamma_x$ from $X$ to $\mathbb{T}$ can be found such that $f_x^{\varphi \circ \alpha_{\lambda_x} \circ \varphi^{-1} \circ \alpha_{\gamma_x}  \circ \varphi}(0) = 0$ for all $x\in D \subset \Fix(\tau)$ where $D$ is a closed set for which $x \mapsto |w_x|^2$ is continuous on $D$.

\begin{proposition} \label{proposition:continuous-extensions}
Let $(\sigma,w)$ and $(\tau,u)$ be WPS on compact $X$, and let $\varphi:\tensor(\sigma,w) \rightarrow \tensor(\tau,u)$ be an algebraic isomorphism covering $X$. If $D \subset \Fix(\tau)$ is a closed on which $x \mapsto |w_x|^2$ is continuous and \emph{non-zero}, there exist continuous functions $\lambda , \gamma : X \rightarrow \mathbb{T}$ such that for all $x \in D$ we have
\begin{equation} \label{eq:uniqueness-prob}
(\hat{f}^{\varphi}_x \circ \hat{f}^{\alpha_{\gamma_x}}_x \circ \hat{f}^{\varphi^{-1}}_x \circ \hat{f}^{\alpha_{\lambda_x}}_x \circ \hat{f}^{\varphi}_x)(0) = 0
\end{equation}
\end{proposition}

\begin{proof}

First note that since the map $x\mapsto |w_x|^2$ is continuous on $D$, and since $|w_x|^2 < 1$ for all $x\in D$, we may extend it to a continuous function $h: X \rightarrow [0,1]$ so that $\|h\|_{\infty} < 1$ still. Next, we simplify equation \ref{eq:uniqueness-prob} to the following equivalent form.
$$
(\hat{f}^{\alpha_{\lambda_x}}_x \circ \hat{f}^{\varphi}_x)(0) = (\hat{f}^{\varphi}_x \circ \hat{f}^{\alpha_{\overline{\gamma_x}}}_x \circ \hat{f}^{\varphi^{-1}}_x)(0)
$$
Which is equivalent to have for all $x\in D$ that
\begin{equation} \label{eq:equiv-form}
\lambda_x e^{i\theta_x}w_x = \hat{f}^{\varphi}_x(\overline{\gamma_x}w_x) = e^{i\theta_x}\frac{w_x - \overline{\gamma_x}w_x}{1-\overline{\gamma_x}|w_x|^2}
\end{equation}
It then suffices to find continuous functions $\gamma, \lambda : X \rightarrow \mathbb{T}$ such that for any $x\in D$,

\begin{equation} \label{eq:lambda-gamma-connection}
\lambda_x = \frac{1-\overline{\gamma_x}}{1-\overline{\gamma_x}h(x)} = \frac{\gamma_x - 1}{\gamma_x - h(x)}
\end{equation}
as multiplying both sides by $e^{i\theta_x}w_x$ yields equation \ref{eq:equiv-form} for all $x\in D$.

Since $h(x)<1$ for all $x\in X$, we see that $\gamma_x - h(x) \neq 0$ for all $x\in X$, so we may define
$$
\gamma_x = \Big(\frac{1 + h(x)}{2}, \sqrt{1 - \big(\frac{1 + h(x)}{2} \big)^2} \Big)
 \ \ \text{and} \ \
\lambda_x = \frac{\gamma_x - 1}{\gamma_x - h(x)}
$$
As $|\gamma_x - 1| = |\gamma_x - h(x)|$ for all $x\in X$, we see that $\gamma$ and $\lambda$ are well-defined continuous functions from $X$ into $\mathbb{T}$ satisfying equation \ref{eq:lambda-gamma-connection}, and we are done.
\end{proof}

Finally, we are at the point where we can prove the main reduction of this paper, that reduces general isomorphism problems to corresponding semi-graded isomorphism problems. Recall that for a $d'$-variable WPS $(\tau,u)$ on $X$ and an index set $I \subset \{1,...,d'\}$ we defined the coinciding set of $I$ to be
$$
C(I) = \{ \ x\in \cap_{i\in I}X_i \ | \ \tau_i(x) = \tau_j(x) \ \}
$$
where for each $1 \leq i \leq d'$ we have $\tau_i : X_i \rightarrow X$.

\begin{theorem} \label{theorem:from-general-to-semi-graded}
Let $(\sigma,w)$ and $(\tau,u)$ be $d$-variable and $d'$-variable WPS respectively, on the same compact space $X$, and let $\varphi: \tensor(\sigma,w) \rightarrow \tensor(\tau,u)$ be a bounded / isometric isomorphism covering $X$. Then there exists a semi-graded bounded / isometric isomorphism $\psi :\tensor(\sigma,w) \rightarrow \tensor(\tau,u)$ covering $X$ respectively.
\end{theorem}

\begin{proof}
Suppose $k \geq 0$ for which there exists $\psi$ such that $f_x^{\psi}(0) = 0$ for all $x\in \Fix(\tau) \cap \bigcup_{|I| \geq k+1}C(I)$ where we range over all subsets $I\subset \{1,...,d'\}$ of size at least $k+1$. Our assumptions guarantee that such a $k$ exists and that $k \leq d'$, since $\varphi$ certainly satisfies $f_x^{\varphi}(0) = 0$ for all $x\in \Fix(\tau) \cap \bigcup_{|I| \geq d'+1}C(I) = \emptyset$.

If there exists $\psi$ for which we can take $k=0$, then $f_x^{\psi}(0) = 0$ for all $x\in \Fix(\tau) \cap \bigcup_{I \subset \{1,...,d'\}}C(I) = \Fix(\tau)$, then $\psi$ is semi-graded by Proposition \ref{proposition:semi-graded-character-space}, and we will be done.

Thus, suppose $\varphi$ a bounded / isometric isomorphism and $k > 0$ such that $f_x^{\varphi}(0) = 0$ for all $x\in \Fix(\tau) \cap \bigcup_{|I| \geq k+1}C(I)$ but $f_x^{\varphi}(0) \neq 0$ for some $x\in \Fix(\tau) \cap \bigcup_{|I| \geq k}C(I)$. We will construct $\psi_k$ for which $f_x^{\psi_k}(0) = 0$ for all $x\in \Fix(\tau) \cap \bigcup_{|I| \geq k}C(I)$, so that for $\psi_k$ there is a smaller $k' <k$ for which $f_x^{\psi_k}(0) = 0$ for all $x\in \Fix(\tau) \cap \bigcup_{|I| \geq k'}C(I)$. By successive iterations of this procedure we keep decreasing $k$, so that we would eventually get $\psi$ for which we can take $k=0$, and be done by the previous paragraph.

We claim that under our current assumptions on $\varphi$, on the closed set $D_k = \Fix(\tau) \cap \bigcup_{|I| = k}C(I)$ we have that $x \mapsto |w_x|^2$ is continuous. We know that $\varphi^*(\theta_{x, 0})(W_{(\sigma,w)}) =  f_x^{\varphi}(0) = r_x^w e^{i\theta_x}w_x$ depends on $x \in X$ continuously, so we restrict it to $D_k$. By Proposition \ref{proposition:analytic-discs} we have that $r_x^w = \sqrt{w(x,x)} = \sqrt{\sum_{i : \sigma_i(x) = x \in X_i}w_i(x)}$ and as a function of $x$ is bounded below on $\Fix(\tau) = \Fix(\sigma) \supset D_k$, and is hence non-zero on $D_k$. Moreover, the only discontinuities $x \mapsto r_x^w$ can have on $D_k$ are those arising from branching points $B(J) \cap \Fix(\tau)$ in $D_k$ for subsets $J \supsetneq I$ and $|I| \geq k$, and so $x\mapsto |w_x|^2 = \frac{f_x^{\varphi}(0)^2}{(r_x^w)^2}$ is continuous at every $x\in D_k$ which is not a point in $B(J) \cap \Fix(\tau)$ for some $J\supsetneq I$ and $|I| \geq k$.

Next, for a point $y\in B(J) \cap \Fix(\tau)$ inside $D_k$ for some $J \supsetneq I$ and $|I| \geq k$, our assumptions guarantee that $0 = f_x^{\varphi}(0) = r_x^w e^{i\theta_x}w_x$ for all $x\in C(J)$ since $|J| > |I| \geq k$, so that $|w_y|^2 = 0$, since $x \mapsto r_x^w$ is non-zero for all $x \in D_k$. 

Now, since $x \mapsto r_x^w$ is bounded below on $D_k$, say by $\epsilon$, we have that $|f_x^{\varphi}(0)|^2 \geq \epsilon^2 |w_x|^2$, and by continuity of $x \mapsto f_x^{\varphi}(0)$ at $y$, we see that $|w_x|^2 \rightarrow 0$ as $x \rightarrow y$. This means that $x \mapsto |w_x|^2$ is continuous at $y$ inside $D_k$, so that $x \mapsto |w_x|^2$ is continuous on all of $D_k$.

Using Proposition \ref{proposition:continuous-extensions} we have two continuous maps $L: x \mapsto \lambda_x$ and $G: x\mapsto \gamma_x$ from $X$ to $\mathbb{T}$ that satisfy equation \ref{eq:uniqueness-prob} for any $x\in D_k$.
Define two unitaries $U_L$ on $C(\sigma,w)$ and $U_G$ on $C(\tau,u)$ given by $U_L(\xi) = L \cdot \xi$ and $U_G(\eta) = G \cdot \eta$ for $\xi \in C(\sigma,w)$ and $\eta \in C(\tau,u)$ using the left action by continuous functions. Next, use Theorem \ref{theorem:isometric-graded} to promote $U_L$ and $U_G$ to (completely) isometric graded automorphisms $Ad_{U_L} : \tensor(\sigma,w) \rightarrow \tensor(\sigma,w)$ and $Ad_{U_G} : \tensor(\tau,u) \rightarrow \tensor(\tau,u)$ such that for any point $x\in D_k$ we have
$$
\hat{f}_x^{Ad_{U_L}}(z) = \lambda_x z  \ \ \text{and} \ \ \hat{f}_x^{Ad_{U_G}}(z) = \gamma_x z
$$
Thus, we get for all $x\in D_k$ that $\hat{f}_x^{Ad_{U_L}}(z) = \hat{f}_x^{\alpha_{\lambda_x}}(z)$ and $\hat{f}_x^{Ad_{U_G}}(z) = \hat{f}_x^{\alpha_{\gamma_x}}(z)$.
Next, we define $\psi_k = \varphi \circ Ad_{U_L} \circ \varphi^{-1} \circ Ad_{U_G} \circ \varphi : \tensor(\sigma,w) \rightarrow \tensor(\tau,u)$, and since for every $x\in D_k$,
$$
\hat{f}_x^{\psi_k}(0) = \hat{f}_x^{\varphi \circ Ad_{U_L} \circ \varphi^{-1} \circ Ad_{U_G}  \circ \varphi}(0) = (\hat{f}_x^{\varphi} \circ \hat{f}_x^{Ad_{U_G}} \circ \hat{f}_x^{\varphi^{-1}} \circ \hat{f}_x^{Ad_{U_G}} \circ \hat{f}_x^{\varphi}) (0) =
$$
$$
(\hat{f}_x^{\varphi} \circ \hat{f}_x^{\alpha_{\gamma_x}} \circ \hat{f}_x^{\varphi^{-1}} \circ \hat{f}_x^{\alpha_{\lambda_x}} \circ \hat{f}_x^{\varphi}) (0) = 0
$$
we obtain that $\psi_k$ is a bounded / isometric isomorphism (respectively to what $\varphi$ is) such that $f_x^{\psi_k}(0) = 0$ for all $x\in D_k = \Fix(\tau) \cap \bigcup_{|I| \geq k}C(I)$, and we have managed to find $\psi_k$ for which we can take $k' < k$ such that $f_x^{\psi_k}(0) = 0$ for all $x\in \Fix(\tau) \cap \bigcup_{|I| \geq k'}C(I)$.
\end{proof}

\subsection{Main results}
The following are our main theorems that resolve algebraic / bounded / isometric isomorphism problems, and classify tensor algebras arising from WPS up to bounded / isometric isomorphisms.

\begin{theorem}[Algebraic / Bounded isomorphisms] \label{theorem:alg-bdd-main}
Let $(\sigma,w)$ and $(\tau,u)$ be WPS over $X$ and $Y$ respectively. The following are equivalent
\begin{enumerate}
\item
$(\sigma,w)$ and $(\tau,u)$ are weighted-orbit conjugate.
\item
$C(\sigma,w)$ and $C(\tau,u)$ are similar.
\item
There exists a graded completely bounded isomorphism $\varphi : \tensor(\sigma,w) \rightarrow \tensor(\tau,u)$.
\item
There exists a bounded isomorphism $\varphi : \tensor(\sigma,w) \rightarrow \tensor(\tau,u)$.
\end{enumerate}
Moreover, if either $\sigma$ or $\tau$ are well-supported, the above is equivalent to the existence of an algebraic isomorphism $\varphi : \tensor(\sigma,w) \rightarrow \tensor(\tau,u)$.
\end{theorem}

\begin{proof}
The equivalence of between (1) and (2) and (3) is given by Proposition \ref{prop:weighted-orbit-char} and Theorem \ref{theorem:bounded-graded}, and (3) implies (4) trivially. To show that (4) implies (3), let $\varphi : \tensor(\sigma,w) \rightarrow \tensor(\tau,u)$ be a bounded isomorphism. By Theorem \ref{theorem:from-general-to-semi-graded}, there is a semi-graded bounded isomorphism $\psi : \tensor(\sigma,w) \rightarrow \tensor(\tau,u)$. Then by Proposition \ref{proposition:semi-graded-to-graded} we obtain a completely bounded graded isomorphism $\widetilde{\psi} : \tensor(\sigma,w) \rightarrow \tensor(\tau,u)$.

For the last part, if either $\sigma$ or $\tau$ are well-supported, and $\varphi : \tensor(\sigma,w) \rightarrow \tensor(\tau,u)$ is an algebraic isomorphism, then by Theorem \ref{theorem:automatic-continuity}, either $\varphi$ or $\varphi^{-1}$ is bounded, but then by the Open Mapping Theorem in Banach spaces, both are bounded.
\end{proof}

\begin{theorem}[Isometric isomorphisms] \label{theorem:isom-main}
Let $(\sigma,w)$ and $(\tau,u)$ be WPS over $X$ and $Y$ respectively. The following are equivalent
\begin{enumerate}
\item
$(\sigma,w)$ and $(\tau,u)$ are branch-transition conjugate.
\item
$C(\sigma,w)$ and $C(\tau,u)$ are unitarily isomorphic.
\item
There exists a graded completely isometric isomorphism $\varphi : \tensor(\sigma,w) \rightarrow \tensor(\tau,u)$.
\item
There exists an isometric isomorphism $\varphi : \tensor(\sigma,w) \rightarrow \tensor(\tau,u)$.
\end{enumerate}
\end{theorem}

\begin{proof}
The equivalence of between (1) and (2) and (3) is given by Corollary \ref{cor:branch-transition-char} and Theorem \ref{theorem:isometric-graded}, and (3) implies (4) trivially. To show that (4) implies (3), let $\varphi : \tensor(\sigma,w) \rightarrow \tensor(\tau,u)$ be an isometric isomorphism. By Theorem \ref{theorem:from-general-to-semi-graded}, there is a semi-graded isometric isomorphism $\psi : \tensor(\sigma,w) \rightarrow \tensor(\tau,u)$. Then by Corollary \ref{proposition:isometric-semi-graded-to-graded} we obtain a completely isometric graded isomorphism $\widetilde{\psi} : \tensor(\sigma,w) \rightarrow \tensor(\tau,u)$.
\end{proof}

\begin{example}  \label{ex:isometric-vs-bounded}
By the above two theorems and Example \ref{example:different-invariants} we see that there are two WPS for which there exists an algebraic / bounded isomorphism of their tensor algebras but there is no \emph{isometric} isomorphism between their tensor algebras. 

This shows that the isometric isomorphism problem for general tensor algebras cannot be solved just by extracting information from algebraic / bounded isomorphism invariants of the tensor algebras, such as representations into upper triangular $2 \times 2$ matrices, which were used in \cite{IsoConjAlg, Multivar, Davidson-Roydor-topo-graphs, IsoDirGrAlg, Solel-Quiver}.
\end{example}

\section{Applications and comparisons} \label{sec:app-comp}

In this section we apply our results to certain subclasses of WPS, by computing what the conjugation relations yield for these classes. For some classes of WPS, our tensor algebras coincide with previously-investigated operator algebras, and we are then able to recover some results on classification of operator algebras, under certain hypotheses on the class.

We note that the results we are able to recover here provide alternative proofs to those proofs using methods of representations into upper triangular $2 \times 2$ matrices as used in \cite{IsoConjAlg, Multivar, Davidson-Roydor-topo-graphs, IsoDirGrAlg, Solel-Quiver}.

\subsection{Non-negative matrices and multiplicity free directed graphs}

When we have a non-negative matrix $A= [A_{ij}]$ indexed by a finite set $\Omega$, we associated a $d$-variable WPS $(\sigma^A, w^A)$ to it as in Subsection \ref{subsec:nnm}, to which we associate a topological quiver $ \mathcal{Q}(A):= \mathcal{Q}(\sigma^A,w^A) = (Gr(A),P(A))$ given by $Gr(A) = Gr(\sigma^A) = \{ \ (i,j) \ | \ A_{ij} >0 \ \}$ with Radon measures $P(A)_j := P(\sigma^A,w^A)_j = \sum_{i\in \Omega} A_{ij}\delta_{(i,j)}$. This topological quiver encodes the information of the entries of $A$ into the Radon measures, since the entries $A_{ij} = \int_{Gr(A)} \chi_{\{(i,j)\}} dP(A)_j$ can be detected by integration against a characteristic function of a singleton. We note that $Gr(A)$ has no sinks if and only if $(\sigma^A,w^A)$ is well-supported. 

Since the tensor algebra associated to the C*-correspondence arising from the topological quiver $\mathcal{Q}(A)$ is $\tensor(\sigma^A,w^A)$, we just write $\tensor(A) : = \tensor(\sigma^A,w^A)$.

Thus, for two non-negative matrices $A=[A_{ij}]$ and $B=[B_{ij}]$ indexed by $\Omega^A$ and $\Omega^B$ respectively, we see that the graphs $Gr(A)$ and $Gr(B)$ are isomorphic directed graphs if and only if $(\sigma^A,w^A)$ and $(\sigma^B,w^B)$ are graph conjugate, if and only if they are branch-transition conjugate, since the topology on $Gr(\sigma^A) = Gr(A) = \{ \ (i,j) \ | \ A_{ij} > 0 \ \}$ is discrete, and so the weight-transition functions will always be continuous. We hence obtain the following:

\begin{corollary}
Let $A$ and $B$ be non-negative matrices indexed by a finite set $\Omega$. Then
$Gr(A)$ and $Gr(B)$ are isomorphic directed graphs if and only if
$\tensor(A)$ and $\tensor(B)$ are (completely) bounded / (completely) isometrically isomorphic. Moreover, if either $Gr(A)$ or $Gr(B)$ have no sinks, the above is equivalent to $\tensor(A)$ and $\tensor(B)$ being algebraically isomorphic.
\end{corollary}

When the non-negative matrix is given as the incidence matrix $A_E = [m_{w,v}]$ of some finite directed graph $G = (V,E,r,s)$ as in Subsection \ref{subsec:fdg}, such that $m_{w,v}$ is either $0$ or $1$, then $G$ is multiplicity free. The topological quiver $\mathcal{Q}(A_E)$ associated to $A_E$ is then just the topological quiver structure we associate to the original graph $G$, that is $G = Gr(A_E) : = \{ \ (r(e),s(e)) = e \ | \ e \in E \ \}$, with Radon measures given by counting measure $P(G)_v = P(A_E)_v = \sum_{w : (w,v) \in Gr(A_E)} \delta_{(w,v)}$ on $s^{-1}(v)$ (See Example 3.19 in \cite{Muhly-Tomforde}, with reversed source and range maps). This means that the tensor algebra $\tensor(G): = \tensor(C(G))$ associated to $G$ as in \cite{IsoDirGrAlg}, coincides with $\tensor(C(Q(A_E)))$, and we recover results of Katsoulis and Kribs in \cite{IsoDirGrAlg} and Solel in \cite{Solel-Quiver}, for the case of finite multiplicity-free graphs.

\begin{corollary} \label{cor:graph-mf}
Let $G$ and $G'$ be \emph{finite multiplicity free} graphs. Then $G$ and $G'$ are isomorphic as directed graphs if and only if $\tensor(G)$ and $\tensor(G')$ are (completely) bounded / (completely) isometrically isomorphic. Moreover, if either $G$ or $G'$ have no sinks, the above is equivalent to $\tensor(G)$ and $\tensor(G')$ being algebraically isomorphic.
\end{corollary}

\subsection{Peters' semi-crossed product}

For a continuous map $\sigma$ on a compact space $X$, we can associate an operator algebra $C(X) \times _{\sigma} \mathbb{Z}_+$ called \emph{Peters' semi-crossed product} to it as done originally by Peters in \cite{Peters-original}. We do this here by giving a universal definition. We call a pair $(\rho, T)$ a representation of $(X,\sigma)$ if $\rho: C(X) \rightarrow B(H)$ is a *-representation and $S \in B(H)$ a contraction such that $ \rho(f) S = S \rho (f \circ \sigma)$. We say that a $(\rho,S)$ is \emph{isometric} if in addition $S$ is an isometry.

Peters' semi-crossed product $C(X) \times _{\sigma} \mathbb{Z}_+$ of the system $(X,\sigma)$ is the norm closed algebra generated by the image of a universal \emph{isometric} representation $(\rho_u, S_u)$ for $(X,\sigma)$. Note that for any isometric representation $(\rho, S)$ we have that $S^* \rho(f) S = \rho(f\circ \sigma)$ and we obtain a representation of the WPS $(\sigma,1)$ as in Definition \ref{defi:representation-of-WPS}.

Muhly and Solel show in \cite{Muhly-Solel-Tensor-Rep} that every representation $(\rho,S)$ of $(X,\sigma)$ dilates to an isometric representation, so that Peters' semi-crossed product is also the norm closed algebra generated by the image of a (contractive) universal representation $(\rho_u, S_u)$.

When we look at $\sigma = (\sigma,1)$ as a WPS, by Proposition 3.21 in \cite{ExelCroPro} any representation $(\pi,T)$ of the WPS $(\sigma,1)$ satisfies $\pi(f) T = T \pi(f\circ \sigma)$, so in fact we have obtained a representation of the system $(X,\sigma)$ as in the sense above. We conclude that $\tensor(\sigma,1) \cong C(X) \times_{\sigma} \mathbb{Z}_+$.

On the other hand, for two continuous maps $\sigma$ and $\tau$ on compact spaces $X$ and $Y$ respectively, we have that $(\sigma,1)$ and $(\tau,1)$ are graph-conjugate if and only if $\sigma$ and $\tau$ are conjugate, and we obtain the following alternative proof, assuming our spaces are compact, to a theorem first proven by Davidson and Katsoulis as Corollary 4.7 in \cite{IsoConjAlg} via methods of representations into triangular $2 \times 2$ matrices.

\begin{corollary} \label{cor:conjugacy-problem}
Let $\sigma$ and $\tau$ be continuous maps on compact spaces $X$ and $Y$ respectively. Then $C(X) \times _{\sigma} \mathbb{Z}_+$ and $C(X) \times _{\sigma} \mathbb{Z}_+$ are algebraically / (completely) bounded /  (completely) isometrically isomorphic if and only there is some homeomorphism $\gamma : X \rightarrow Y$ such that $\gamma^{-1} \tau \gamma = \sigma$.
\end{corollary}

\subsection{Partial systems and topological graphs}

When $\sigma = (\sigma_1,...,\sigma_d)$ is a partially defined system so that $\sigma_i : X_i \rightarrow X$, we think of it as a WPS by specifying $w=(1,...,1)$. The weight induced on $Gr(\sigma)$ then becomes $m_{\sigma}(e):= w(e) = \sum_{i\in I(e,\sigma)}1 = |I(e,\sigma)|$ is just the multiplicity of $e\in Gr(\sigma)$. In this case, we have the following characterization of isometric isomorphism between tensor algebras. Denote by $\tensor(\sigma) : = \tensor(\sigma,1)$ the tensor algebra of a partially defined system $\sigma$.

\begin{theorem}
Let $\sigma$ and $\tau$ be partially defined systems on compact $X$ and $Y$ respectively. Then $\tensor(\sigma)$ and $\tensor(\tau)$ are isometrically isomorphic if and only if there is a homeomorphism $\gamma : X \rightarrow Y$ such that $Gr(\sigma) = Gr(\tau^{\gamma})$ and the function $\frac{m_{\tau^{\gamma}}}{m_{\sigma}}: Gr(\sigma) \rightarrow (0,\infty)$ is locally constant.
\end{theorem}

\begin{proof}
By Theorem \ref{theorem:isom-main} we have that $\tensor(\sigma)$ and $\tensor(\tau)$ are isometrically isomorphic if and only if there is a homeomorphism $\gamma : X \rightarrow Y$ such that $Gr(\sigma) = Gr(\tau^{\gamma})$ and the function $\frac{m_{\tau^{\gamma}}}{m_{\sigma}}: Gr(\sigma) \rightarrow (0,\infty)$ is continuous. Since $\frac{m_{\tau^{\gamma}}}{m_{\sigma}}$ can only attain finitely many values, it is continuous if and only if it is locally constant.
\end{proof}

One can also provide a similar theorem characterizing algebraic / bounded isomorphisms using Theorem \ref{theorem:alg-bdd-main}, in terms of multiplicity functions and orbits of the WPS.

Next, we would like to compare our algebras with an algebra constructed from topological graphs. For a partial system $\sigma = (\sigma_1,...,\sigma_d)$ on a compact $X$ so that $\sigma_i : X_i \rightarrow X$ for $X_i$ clopen, we can associate a topological graph $\mathcal{E}(\sigma):= (X, Gr^{dis}(\sigma), r,s)$ to it in the sense of Katsura \cite{Katsura-topological-graphs}. We define the edges to be the \emph{disjoint union}
$$
Gr^{dis}(\sigma) = \sqcup_{i=1}^d Gr(\sigma_i) = \sqcup_{i=1}^d \{ \ (\sigma_i(x),x) \ | \ x\in X_i \ \}
$$
and range and source maps are given by $r(\sigma_i(x),x) = \sigma_i(x)$ and $s(\sigma_i(x),x) = x$. This way $s$ is a local homeomorphism, $r$ is continuous, and we obtain a topological graph. 

From Example 3.20 of \cite{Muhly-Tomforde} (with reversed source and range) we see that we can think of $\mathcal{E}(\sigma)$ as a topological quiver as follows: since for every $x\in X$ we have that $s^{-1}(x)$ is discrete, we may define counting measures $P(\sigma)_x$ on $s^{-1}(x)$, so as to obtain a topological quiver in the sense of Muhly and Tomforde. When we associate a C*-correspondence to this topological quiver, as we do in the discussion after Definition \ref{defi:topological-quiver}, it coincides with the C*-correspondence that one usually associates to the topological graph as done by Katsura in \cite{Katsura-topological-graphs}, and we denote this C*-correspondence by $C(\mathcal{E}(\sigma))$. 

When we take the tensor algebra $\tensor(\mathcal{E}(\sigma)):=\tensor(C(\mathcal{E}(\sigma)))$ we obtain the tensor algebra associated to the topological graph of $\sigma$. Tensor algebras associated to topological graphs were investigated by Davidson and Roydor in \cite{Davidson-Roydor-topo-graphs}, where they show that for two (compact) topological graphs, if the tensor algebras associated to them are algebraically isomorphic, then the topological graphs are locally conjugate as in the following definition. 

\begin{defi}
Let $E = (E^0, E^1,s_E, r_E)$ and $F = (F^0,F^1,s_F,r_F)$ be two compact topological graphs. They are said to be \emph{locally conjugate} if there exists a homeomorphism $\gamma : E^0 \rightarrow F^0$ such that for any $x \in E^0$, there is a neighborhood $U$ of $x$ and a homeomorphism $\lambda: s^{-1}_E(U) \rightarrow s^{-1}_F(\gamma U)$ such that $s_F \lambda = \gamma s_E |_{s^{-1}_E(U)}$ and $r_F \lambda = \gamma r_E |_{s^{-1}_E(U)}$.
\end{defi}

A natural question to ask is, when does the topological quiver $\mathcal{Q}(\sigma): = \mathcal{Q}(\sigma,1)$ given in Definition \ref{def:quiver-and-positive-op} coincide with $\mathcal{E}(\sigma)$? 

In general, $\mathcal{Q}(\sigma)$, does not coincide with $\mathcal{E}(\sigma)$, as $\mathcal{Q}(\sigma)$ loses some of the information on the multiplicity of edges. However, if $\sigma$ has no coinciding points, that is, if for all $x\in X$ we have $\sigma_i(x) \neq \sigma_j(x)$ for all $i\neq j$ in $\{1,...,d\}$ and $x\in X_i \cap X_j$, then $Gr^{dis}(\sigma) = Gr(\sigma)$ is multiplicity free so that $\mathcal{Q}(\sigma)$ and $\mathcal{E}(\sigma)$ coincide as topological graphs.

Thus, in the case when there are no coinciding points, the correspondences $C(\mathcal{Q}(\sigma))$ and $C(\mathcal{E}(\sigma))$ coincide, and so do their tensor algebras $\tensor(\sigma)$ and $\tensor(\mathcal{E}(\sigma))$ respectively. We then obtain the following classification result for topological graphs arising from partial systems with no coinciding points, which yield a class of multiplicity free topological graphs. See Subsection 8.3 of \cite{Kakariadis-Shalit} where $Q$-$P$ local piecewise conjugacy of a partial system is compared with local conjugacy of associated topological graphs.

\begin{corollary} \label{cor:multip-free-part-sys}
Let $\sigma$ and $\tau$ be $d$ and $d'$ variable partial systems over compact $X$ and $Y$ respectively, and suppose that \emph{both} $\sigma$ and $\tau$ have no coinciding points. Then $\tensor(\mathcal{E}(\sigma))$ and $\tensor(\mathcal{E}(\tau))$ are (completely) bounded / (completely) isometric if and only if $\sigma$ and $\tau$ are graph conjugate, if and only if $\mathcal{E}(\sigma)$ and $\mathcal{E}(\tau)$ are locally conjugate. If either $\sigma$ or $\tau$ are well-supported, the above is equivalent to $\tensor(\mathcal{E}(\sigma))$ and $\tensor(\mathcal{E}(\tau))$ being algebraically isomorphic.
\end{corollary}

\begin{proof}
From Corollary \ref{cor:no-branching-graph} and Theorems \ref{theorem:alg-bdd-main} and \ref{theorem:isom-main} we see that $\tensor(\mathcal{E}(\sigma))$ and $\tensor(\mathcal{E}(\tau))$ are (completely) bounded / (completely) isometrically isomorphic if and only if $\sigma$ and $\tau$ are graph conjugate, so we need only show that graph conjugacy of $\sigma$ and $\tau$ is equivalent to local conjugacy of $\mathcal{E}(\sigma)$ and $\mathcal{E}(\tau)$.

Suppose $\mathcal{E}(\sigma)$ and $\mathcal{E}(\tau)$ are locally conjugate via $\gamma : E^0 \rightarrow F^0$. Since $E^1$ and $F^1$ can be identified with $Gr(\sigma)$ and $Gr(\tau)$ respectively, it is easy to see that $\gamma \times \gamma : Gr(\sigma) \rightarrow Gr(\tau)$ is a homeomorphism so that $Gr(\sigma) = Gr(\tau^{\gamma})$.

For the converse, suppose that $Gr(\sigma) = Gr(\tau^{\gamma})$ via some homeomorphism $\gamma : X \rightarrow Y$. We may assume without loss of generality that $\gamma = Id_X$ since we know that $F^{\gamma} = \Ee(\tau^{\gamma})$ and $F = \Ee(\tau)$ are isomorphic topological graphs via the pair $(\gamma, \gamma \times \gamma)$ in the sense that $\Gamma: = \gamma \times \gamma : Gr(\tau^{\gamma}) \rightarrow Gr(\tau)$ is a homeomorphism and $\gamma s_{F^{\gamma}} = s_F \Gamma$ and $\gamma r_{F^{\gamma}} = r_F \Gamma$, and isomorphism of topological graphs implies local conjugacy of the graphs. Let $s_{\sigma}$ and $r_{\sigma}$ denote the source and range maps of $\Ee(\sigma)$, and similarly denote $s_{\tau}$ and $r_{\tau}$ for the source and range of $\Ee(\tau)$.

Thus, for each $x \in X$ we have that $\{ \sigma_i(x) \}_{x\in X^{\sigma}_i} = \{\tau_j(x) \}_{x\in X^{\tau}_j}$ where $\sigma_i : X^{\sigma}_i \rightarrow X$ and $\tau_j : X^{\tau}_j \rightarrow X$. Since both $\sigma$ and $\tau$ have no coinciding points, we have that the tuples $( \sigma_i(x) )_{x\in X^{\sigma}_i}$ and $( \tau_j(x) )_{x\in X^{\tau}_j}$ are each comprised of pairwise distinct elements, and we have that $D^{\sigma}_x: = \{ i | x\in X^{\sigma}_i \}$ and $D^{\tau}_x : = \{ j | x \in X^{\tau}_j\}$ have the same cardinality. Hence, there is a bijection $\alpha_x : D^{\sigma}_x \rightarrow D^{\tau}_x$ so that $\sigma_i(x) = \tau_{\alpha_x(i)}(x)$. Now fix an $x\in X$. Continuity of $\sigma$ and $\tau$ guarantee that there exists an open neighborhood $V$ of $x$ such that $D^{\sigma}_y = D^{\sigma}_x$ and $D^{\tau}_y = D^{\tau}_x$ for all $y\in V$, and we obtain a continuous map $\alpha: y \mapsto \alpha_y$ from $V$ to the set of bijections between $D^{\sigma}_x$ and $D^{\tau}_x$. Take $U = \{ y \in V | \sigma_i(y) = \tau_{\alpha_x(i)}(y) \} $ which is clopen in $V$, and is hence open in $X$. We then define a homeomorphism $\lambda : s^{-1}_{\sigma}(U) \rightarrow s^{-1}_{\tau}(U)$ by setting $\lambda(\sigma_i(y),y) = (\tau_{\alpha_x(i)}(y),y)$ which satisfies $s_{\tau} \lambda = s_{\sigma} |_{s^{-1}_{\sigma}(U)}$ and $r_{\tau} \lambda = r_{\sigma} |_{s^{-1}_{\sigma}(U)}$ as required.
\end{proof}

\subsection*{Acknowledgments} The author would like to express his gratitude to Ken Davidson, his Ph.D. adviser, for the many fruitful discussions and keen advice he has provided, which helped refine this work to its current form. The author would also like to thank Orr	Shalit for providing helpful feedback on a draft version of this paper.

\end{document}